\documentclass[10pt,reqno]{amsart}

\usepackage{amssymb}
\usepackage{amsmath}
\usepackage{amsfonts, dsfont}

\usepackage{amsthm} 
\usepackage[utf8]{inputenc} 
\usepackage{enumitem} 

\usepackage{graphicx}

\usepackage{mathrsfs}

\usepackage[utf8]{inputenc}
\usepackage{anyfontsize}

\usepackage{tikz}
\usetikzlibrary{calc,intersections,arrows.meta,patterns}

\DeclareGraphicsExtensions{.pdf,.png,.jpg}

\usepackage{color}

\theoremstyle{theorem} 
\newtheorem{thm}{Theorem}[section]
\newtheorem{corollary}[thm]{Corollary}

\newtheorem{lemma}[thm]{Lemma}
\newtheorem{prop}[thm]{Proposition}
\newtheorem{assumption}[thm]{Assumption}

\theoremstyle{definition}   
\newtheorem{defn}[thm]{Definition}

\theoremstyle{remark}  
\newtheorem{remark}[thm]{Remark}

\newcommand\bK{\mathbb{K}}
\newcommand\bL{\mathbb{L}}

\newcommand\bR{\mathbb{R}}

\newcommand\bH{\mathbb{H}}
\newcommand\bZ{\mathbb{Z}}

\newcommand\bN{\mathbb{N}}

\newcommand\cA{\mathcal{A}}
\newcommand\cB{\mathcal{B}}
\newcommand\cC{\mathcal{C}}
\newcommand\cD{\mathcal{D}}

\newcommand\cK{\mathcal{K}}
\newcommand\cL{\mathcal{L}}

\newcommand\cM{\mathcal{M}}
\newcommand\cT{\mathcal{T}}
\newcommand\cO{\mathcal{O}}

\newcommand{\domain}{\mathcal{O}}

\newcommand{\mysection}[1]{\section{#1}
\setcounter{equation}{0}}

\begin{document}

\title[Parabolic equations on conic domains]
{A weighted Sobolev regularity theory  of  the  parabolic equations with measurable coefficients
 on conic domains in $\bR^d$}
 
\thanks{The first and  third authors were  supported by the National Research Foundation of Korea(NRF) grant funded by the Korea government(MSIT) (No. NRF-2020R1A2C1A01003354)} 
\thanks{The second author was  supported by the National Research Foundation of Korea(NRF) grant funded by the Korea government(MSIT) (No. NRF-2019R1F1A1058988)}

\author{Kyeong-Hun Kim}
\address{Kyeong-Hun Kim, Department of Mathematics, Korea University, Anam-ro 145, Seongbuk-gu, Seoul, 02841, Republic of Korea}
\email{kyeonghun@korea.ac.kr}

\author{Kijung Lee}
\address{Kijung Lee, Department of Mathematics, Ajou University, Worldcup-ro 206, Yeongtong-gu, Suwon, 16499, Republic of Korea}
\email{kijung@ajou.ac.kr}

\author{jinsol Seo}
\address{Jinsol Seo, Department of Mathematics, Korea University, Anam-ro 145, Seongbuk-gu, Seoul, 02841, Republic of Korea}
\email{seo9401@korea.ac.kr}



\begin{abstract}
We establish existence, uniqueness, and arbitrary order Sobolev regularity results for the second order parabolic equations with measurable coefficients defined on the conic domains $\cD$ of the type
\begin{equation}
\label{eqn ab}
\cD(\cM):=\left\{x\in \bR^d :\,\frac{x}{|x|}\in \cM\right\}, \quad \quad \cM \subset S^{d-1}.
\end{equation}
We obtain the regularity results   by using
 a system of mixed weights consisting of appropriate powers of the distance to the vertex and of the distance to the boundary.
  We also provide the sharp  ranges of admissible powers of the distance to the vertex and  to the boundary.
\end{abstract}

\maketitle

\mysection{Introduction}\label{sec:Introduction}

Various weighted Sobolev spaces have been used in the  study of   elliptic and parabolic  equations,  for instance,  when the equations  are 
defined on non-smooth domains (e.g. \cite{G,KK2004, Kozlov Nazarov 2014, Ku,Sol2001}), they
 are degenerate near the boundary of domains (e.g. \cite{Kim2007,LM}), 
 or   they have rough external forces  (e.g. \cite{KK2004,Krylov 1999-1, Lo1}).  
 In general, such irregular conditions combined with Dirichlet boundary condition cause the derivatives of solutions to blow up at the boundary and consequently one needs appropriate weights to \emph{understand} the blow-up behaviors in view of regularity theory.

In this article we study the weighted Sobolev theory of the parabolic equation
      \begin{equation}
          \label{main eqn intro}
          u_t=\sum_{i,j=1}^d a^{ij}(t)u_{x^ix^j}+f, \quad t>0\quad ; \quad u(0,\cdot)=0
                    \end{equation}
      given with zero boundary condition  on the conic domain  
      $$\cD=\cD(\cM):=\left\{x\in \bR^d :\,\frac{x}{|x|}\in \cM\right\}.
      $$  Here, $\cM$ decides the shape of the domain and we assume that it is an open subset of $S^{d-1}$ with $C^2$ boundary; see Figure 1 in Section \ref{sec:Cone}. The key point of this article is considering such domains which have smooth part and non-smooth part together.
      We also assume that the coefficients $a^{ij}(t)$  are  merely measurable in $t$ and the external force $f$  can be very wild near the boundary of the domain. 
      
      Our interest on conic domains comes from the related theory of stochastic partial differential equations (SPDEs), especially stochastic parabolic equations. In this case the derivatives of the solutions are more sensitive near the boundary than the deterministic case, even near the smooth part of the boundary.  We will give more explanation on this  below. 
      
      To understand the behaviors  of solutions  near the boundary of  conic domains, we use a new weight which  actually  is a unification of weights from the sources  \cite{Krylov 1999-1,Krylov 1999-3, Lo1} and \cite{Kozlov Nazarov 2014,Sol2001,So}. 
    We obtain  regularity results using a system of mixed weights consisting of appropriate powers of the distance to the vertex and of the distance to the boundary.  Let 
   $$
   \rho_{\circ}(x):=|x| \quad  \text{and} \quad  \rho(x):=d(x,\partial \cD)
   $$
    denote the distance from $x\in \cD$ to the vertex and to the boundary of $\cD$, respectively.  We prove in this article that for any $p\in (1,\infty)$ and $n=0,1,2,\cdots$, the estimate   
      \begin{eqnarray}
\nonumber
&& \int^T_0 \int_{\cD} \left(|\rho^{-1}u|^p+|u_x|^p+|\rho u_{xx}|^p+\cdots+ |\rho^{n+1}D^{n+2}u|^p\right) \rho_{\circ}^{\theta-\Theta}\rho^{\Theta-d}\, dx\,dt \\
&& \quad \quad \leq N 
  \int^T_0 \int_{\cD} \left( |\rho f|^p+\cdots+|\rho^{n+1}D^nf|^p\right) \rho_{\circ}^{\theta-\Theta}\rho^{\Theta-d}\, dx\,dt
 \label{main estimate intro}
\end{eqnarray}
holds for the solution to equation \eqref{main eqn intro}  provided  that
  \begin{equation}
     \label{theta con intro}         
d-1<\Theta<d-1+p, \quad\,\,  p(1-\lambda^+_{c})<\theta<p(d-1+\lambda^-_{c}).
 \end{equation}
  Here, $\lambda^+_c$ and $\lambda^-_c$ are positive constants determined by $\cM$   and the operator $\cL=\sum_{ij}a^{ij}(t)D_{ij}$; see  Definition \ref{lambda} below and also see Proposition \ref{prop Theta}.  As can be  seen in  estimate \eqref{main estimate intro},   our mixed weights help us measure   the  regularity of the solution  near both the vertex and  other boundary points. 
    Note also  that  the second and higher derivatives of the  solution  satisfying  \eqref{main estimate intro} are allowed to blow up substantially fast near the boundary. Moreover,  the external force $f$ is allowed to  be very wild near the boundary.

  As we mentioned above, the main motivation of our interest in  the weighted Sobolev spaces  lies in the theory of SPDEs. It turns out (see \cite{KL1999}) that due to the incompatibility between random noises and Dirichlet boundary condition, the second and higher order derivatives of solutions to SPDEs blow up near the boundary; this behavior of the solutions occurs even on $C^{\infty}$ domains. Hence, we need an appropriate weight system  to measure the derivatives near the boundary. In \cite{Kim2004, KL1999} it is shown that if   domains satisfy $C^1$-boundary condition, then the effects of such incompatibility  can be described very accurately by a system of weights based solely on the distance to the boundary. As we may guess, with conic domains we need more subtle approach and it turns out that it is very appropriate for us to involve $\rho$ and $\rho_{\circ}$ in the manner presented in \eqref{main estimate intro}.
  
  A preliminary and important step for the main result of \cite{Kim2004,KL1999}  on SPDEs on $C^1$ domains was constructing the corresponding result on  the deterministic  equation, that is equation \eqref{main eqn intro}. This article is such work related to conic domains and the estimate \eqref{main estimate intro} is the backbone of it. As a comparison, when the boundary is nice, the work was done in   \cite{KK2004,Krylov 1999-1} and the result is as follows: 
 for the solution to  equation \eqref{main eqn intro} defined on a $C^1$  domain $\domain$, it holds that 
  \begin{equation}
     \label{eqn 5.1.1}
   \int^T_0 \int_{\domain} \left(|\rho^{-1}u|^p+|u_x|^p+|\rho u_{xx}|^p\right)\rho^{\Theta-d}\, dx\,dt  \leq N 
  \int^T_0 \int_{\domain}  |\rho f|^p \rho^{\Theta-d}\, dx\,dt
  \end{equation}
  provided that 
  $$
  \quad  p\in(1,\infty), \quad d-1<\Theta<d-1+p.
  $$
We also remark that  if $\partial \cO \in C^{1,\delta}$, $\delta\in (0,1]$, then    the second order derivative of solution can be estimated for wider range of $\Theta$: it is shown in  \cite{Kozlov Nazarov 2009} that 
  $$
  \int^T_0 \int_{\domain} |\rho u_{xx}|^p\rho^{\Theta-d}\, dx\,dt  \leq N 
  \int^T_0 \int_{\domain}  |\rho f|^p \rho^{\Theta-d}\, dx\,dt  
  $$
  holds for $d-1-\delta p<\Theta<d-1+p$. However, smooth domains do not  yield wider range of $\Theta$ for lower order derivatives of solution. That is, estimate \eqref{eqn 5.1.1} holds only for $d-1<\Theta<d-1+p$ even on $C^{\infty}$ domains (see \cite{Krylov 1999-1}).
  
  Now,  what if   domains do not have $C^1$ boundary? For instance,  the boundary could have a vertex which makes the boundary of the domain non-smooth, i.e. a conic domain. Our interest on conic domains arises with a question which, in particular, asks if  there can be an estimate on simple Lipschitz domains that makes estimates  \eqref{eqn 5.1.1}  a particular case for $C^1$  domains. It turns out that estimate \eqref{eqn 5.1.1}  fails to hold in conic domains in general.
Note that \eqref{main estimate intro} with $n=0$ and $\cO=\cD$ may yield \eqref{eqn 5.1.1}  \emph{if} one can take $\theta=\Theta$. However, due to the ranges of $\Theta$ and $\theta$ in \eqref{theta con intro}, taking $\theta=\Theta$ is possible only if 
\begin{equation*}
   \label{eqn 5.1.3}
 \Theta\in (d-1, d-1+p) \, \cap \left(p(1-\lambda^+_c), \, p(d-1+\lambda^-_c)\right).
 \end{equation*}
Actually,  an example in \cite{Kim2014} shows that for any $p>4$, there is a $2$-dimensional conic domain of the type
\begin{equation}\label{wedge domain}
 \cD=\left\{(r\cos \eta, r\sin \eta) \in \bR^2 :\, r>0, \eta\in (-\frac{\kappa}{2},\frac{\kappa}{2})\right\}
\end{equation}
 with an  appropriately chosen $\kappa\in (0,2\pi)$ and a function $u$, a solution to the heat equation,
such that  estimate \eqref{eqn 5.1.1} fails even for $\Theta=d(=2)$, as taking 
$\theta=\Theta$ in \eqref{main estimate intro} is not allowed for the constructed function $u$ in the example. This example demonstrates that the presence of $\rho_{\circ}$ in   \eqref{main estimate intro} is inevitable, and it also suggests that  \eqref{main estimate intro} and \eqref{eqn 5.1.1} are of different character.

In summary, the weight system based only on the distance to the boundary is insufficient to construct a regularity theory of SPDEs defined on general conic domains and one way or another we need a mixed weight like ours described above and we settle down with \eqref{main estimate intro}. In a subsequent article, based on the results presented in this article, we plan to construct the corresponding theory on SPDEs defined on the conic domains.

We also remark that if one formally replaces $\rho_{\circ}$ with $\rho$ in \eqref{main estimate intro},  then  one sees
   \begin{equation}
     \label{eqn 5.1.6}
   \int^T_0 \int_{\cD} \left(|\rho^{-1}_{\circ}u|^p+|u_x|^p+|\rho_{\circ} u_{xx}|^p\right)\rho^{\theta-d}_{\circ}\, dx\,dt  \leq N 
  \int^T_0 \int_{\cD}  |\rho_{\circ} f|^p \rho^{\theta-d}_{\circ}\, dx\,dt,
  \end{equation}
which actually holds true (see \cite{Kozlov Nazarov 2014,Sol2001,So}) for the same  $\theta$ satisfying \eqref{theta con intro}. 
However, the weight system based only on the distance to the vertex provides poor regularity result near the boundary, and moreover it is not much useful  in the study of SPDEs since higher derivatives of solutions to SPDEs can not be controlled without the help of weights related to the distance to the boundary.
Hence, for this article and the subsequent one related to SPDEs, estimate \eqref{main estimate intro} is essential.

 Now, we shortly describe the main steps of the proof for estimate \eqref{main estimate intro}:
\begin{itemize}
\item[-]
 We use a localization argument  to control the higher order derivatives of solution in terms of lower order derivatives of solution and free terms.   Consequently,  the result of this step reduces the problem into obtaining appropriate estimate of the zero-th order derivative of solution.
The idea of our localization argument is taken from \cite{Krylov 1999-1} and  modified in this article to handle  Sobolev spaces with our mixed weights. 

\item[-] We obtain the estimate of the zero-th order derivative of solution using the solution representation formula and a refined Green's function estimate. We use  direct but very delicate computations to derive the desired estimate. Such direct computation skill is borrowed from  \cite{Kozlov Nazarov 2014} and modified here to  handle a mixed weight system.
\end{itemize}

Finally, we would like to add an important comment that   the study on  conic domains with $d> 2$ is much involved than the   case $d=2$.   This article includes this task in Section 3. 

This article is organized as follows. In Section 2 we introduce some properties of weighted Sobolev spaces and   present our main result, Theorem \ref{main result}.  In Section 3 we estimate weighed $L_p$ norm of the zero-th order derivative of the solution  based on direct but  highly nontrivial computations. The estimates of the derivatives of the solution are obtained in Section 4, and finally in Section 5 our main result is proved. 
\vspace{2mm}

\noindent\textbf{Notation}
\begin{itemize}

\item We use $:=$ to denote a definition.

\item  For a measure space $(A, \cA, \mu)$, a Banach space $B$ and $p\in[1,\infty)$, we write $L_p(A,\cA, \mu;B)$ for the collection of all $B$-valued $\bar{\cA}$-measurable functions $f$ such that
$$
\|f\|^p_{L_p(A,\cA,\mu;B)}:=\int_{A} \lVert f\rVert^p_{B} \,d\mu<\infty.
$$
Here, $\bar{\cA}$ is the completion of $\cA$ with respect to $\mu$.  The Borel $\sigma$-algebra on a topological space $E$ is denoted by $\cB(E)$. We will drop $\cA$ or $\mu$ or even $B$ in $L_p(A,\cA, \mu;B)$   when they   are obvious from the context. 

\item $\bR^d$ stands for the $d$-dimensional Euclidean space of points $x=(x^1,\cdots,x^d)$, $B_r(x):=\{y\in \bR^d: |x-y|<r\}$, 
 $\bR^d_+:=\{x=(x^1,\ldots,x^d): x^1>0\}$, and $S^{d-1}:=\{x\in \bR^d: |x|=1\}$.
 
 \item For  $\domain \subset \bR^d$, $B^{\domain}_R(x):=B_R(x)\cap \domain$ and  $Q^{\domain}_R(t,x):=(t-R^2,t]\times B^{\domain}_R(x)$.
 
 \item  $\bN$ denotes the natural number system,  and   $\bZ$ denotes the set of integers.

\item For $x$, $y$ in $\bR^d$,  $x\cdot y :=\sum^d_{i=1}x^iy^i$ denotes the standard inner product.

\item For a domain $\domain$ in $\bR^d$, $\partial \domain$ denotes the boundary of $\domain$.

\item  For  any multi-index $\alpha=(\alpha_1,\ldots,\alpha_d)$, $\alpha_i\in \{0\}\cup \bN$,   
$$
f_t=\frac{\partial f}{\partial t}, \quad f_{x^i}=D_if:=\frac{\partial f}{\partial x^i}, \quad D^{\alpha}f(x):=D^{\alpha_d}_d\cdots D^{\alpha_1}_1f(x).
$$
 We denote $|\alpha|:=\sum_{i=1}^d \alpha_i$.  For the second order derivatives we denote $D_jD_if$ by $D_{ij}f$. We often use the notation 
$|gf_x|^p$ for $|g|^p\sum_i|D_if|^p$ and $|gf_{xx}|^p$ for $|g|^p\sum_{i,j}|D_{ij}f|^p$.  We also use $D^m f$ to denote arbitrary partial derivatives  of order $m$ with respect to the space variable.

\item $\Delta_x f:=\sum_i D_{ii}f$, the Laplacian for $f$.

\item For $n\in \{0\}\cup \bN$,  $W^n_p(\domain):=\{f: \sum_{|\alpha|\le n}\int_{\domain}|D^{\alpha}f|^p dx<\infty\}$, the Sobolev space.

\item  For a domain $\domain \subseteq\bR^d$, $\cC^{\infty}_c(\domain)$ is the the space of infinitely differentiable functions with compact support in $\domain$. $supp(f)$ denotes the support of the function $f$. Also, $\cC^{\infty}(\domain)$ denotes the the space of infinitely differentiable functions in $\domain$.

\item  Throughout the article, the letter $N$denotes a finite positive constant which may have different values along the argument  while the dependence  will be informed;  $N=N(a,b,\cdots)$, meaning that  $N$ depends only on the parameters inside the parentheses.

\item  $A\sim B$ means that there is a constant $N$ independent of $A$ and $B$ such that  $A\leq N B$ and $B\leq N A$.

\item $d(x,\domain)$ stands for the distance between a point $x$ and a set $\domain \subset \bR^d$.

\item $a \vee b :=\max\{a,b\}$ and  $a \wedge b :=\min\{a,b\}$. 

\item $1_U$ is the indicator function on $U$.

\item 
We will use the following sets of functions (see \cite{Kozlov Nazarov 2014}).
\begin{itemize}
\item[-]
$\mathcal{V}(Q^{\mathcal{\domain}}_R(t_0,x_0))$ : the set of functions $u$ defined at least on $Q^{\mathcal{\domain}}_R(t_0,x_0)$ and satisfying
\begin{equation*}
\sup_{t\in(t_0-R^2,t_0]}\|u(t,\cdot)\|_{L_2(B^{\domain}_{R}(x_0))} +\|\nabla u\|_{L_2(Q^{\domain}_{R}(t_0,x_0))}<\infty.\nonumber
\end{equation*}
\item[-] 
$\mathcal{V}_{loc}(Q^{\domain}_R(t_0,x_0))$ : the set of functions $u$ defined at least on $Q^{\domain}_R(t_0,x_0)$ and satisfying
\begin{equation*}
u\in \mathcal{V}(Q^{\domain}_r(t_0,x_0)),  \quad \forall r\in (0,R).\nonumber
\end{equation*}
\end{itemize}
\end{itemize}

\mysection{The main result on conic domains}\label{sec:Cone}

Throughout this article we assume $d\ge 2$. Let $\cM$ be a nonempty open set in $S^{d-1}:=\left\{x\in \bR^d\,:\,|x|=1\right\}$ with $\overline{\cM}^S\neq S^{d-1}$, where 
$\overline{\cM}^S$ is the closure of $\cM$ in $S^{d-1}$. 

We define our conic domain  in $\mathbb{R}^d$ by
$$
\mathcal{D}=\cD(\cM):=\Big\{x\in\mathbb{R}^d\setminus\{0\} \ \Big|  \ \ \frac{x}{|x|}\in \mathcal{M} \Big\}.
$$
For example, when $d=2$, for each fixed  angle $\kappa\in\left(0,2\pi\right)$ we can consider
\begin{equation}\label{wedge in 2d}
\mathcal{D}=\mathcal{D}^{(\kappa)}:=\left\{(r\cos\eta,\ r\sin\eta)\in\mathbb{R}^2 \mid r\in(0,\ \infty),\ -\frac{\kappa}{2}<\eta<\frac{\kappa}{2}\right\}.
\end{equation}

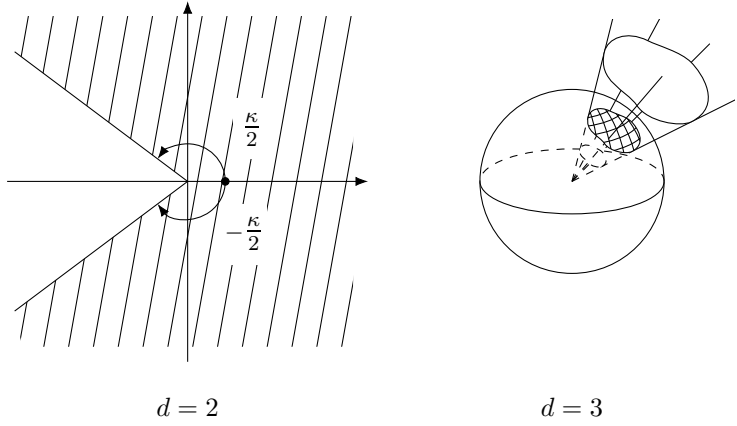
\begin{figure}[ht]
\begin{tikzpicture}[> = Latex]
\begin{scope}
\clip (-2.4,-2.4)--(2.4,-2.4)--(2.4,2.4)--(-2.4,2.4);

\draw[->] (-2.4,0) -- (2.4,0);
\draw[->] (0,-2.4) -- (0,2.4);
\draw (0,0) -- (-2.3,1.725);
\draw (0,0) -- (-2.3,-1.725);
\draw[->] (0.5,0) arc(0:{180-atan(3/4)}:0.5);
\draw[<-] ({-180+atan(3/4)}:0.5) arc({-180+atan(3/4)}:0:0.5);
\begin{scope}
\clip (0,0)--(-2.3,-1.725)--(-2.2,-2.2)--(2.2,-2.2)--(2.2,2.2)--(-2.2,2.2)--(-2.3,1.725)--(0,0);
\foreach \i in {-4,-3.67,...,3}
{\draw (\i,-2.8)--(\i+1,2.8);}
\path[fill=white] (0.60,0.3) rectangle (1.05,1.1);
\path (0.85,0.35) node[above] {{\Large $\frac{\kappa}{2}$}};
\path[fill=white] (0.45,-0.3) rectangle (1.05,-1.1);
\path (0.75,-0.35) node[below] {{\small $-$}{\Large $\frac{\kappa}{2}$}};
\draw[fill=black] (0.5,0) circle (0.5mm);
\end{scope}
\end{scope}
\path (0,-3) node {$d=2$};
\end{tikzpicture}
\begin{tikzpicture}[> = Latex]

\begin{scope}

\begin{scope}[scale=0.8]
\begin{scope}[scale=0.9]
\clip (0,0) circle (3.6) ;
\draw (0,0) circle (1.7);
\draw (1.7,0) arc(0:-180:1.7 and 0.6);
\draw[dashed] (1.7,0) arc(0:180:1.7 and 0.6);
\clip (-3,-3) -- (3,-3) -- (3,3) -- (-3,3);
\draw (1.1,0.55) -- (6,3);
\draw (0.3,1.2) -- (1.5,6);
\draw[dashed] (0,0) -- (1.1,0.55);
\draw[dashed] (0,0) -- (1.1,1.1);
\draw[dashed] (0,0) -- (0.7,1.3);
\draw[dashed] (0,0) -- (0.3,1.2);
\draw (1.1,1.1) -- (3.3,3.3);
\draw (0.7,1.3) -- (2.1,3.9);
\end{scope}

\begin{scope}[scale=0.45]
\draw[fill=white][dashed] (1.1,0.55) .. controls (1.4,0.7) and (1.2,1) .. (1.1,1.1) .. controls (1.0,1.2) and (0.8,1.25) .. (0.7,1.3) .. controls (0.5,1.4) and (0.33,1.32) .. (0.3,1.2) .. controls (0.23,0.92) and (0.5,0.85) .. (0.6,0.7) .. controls (0.7,0.55) and (1,0.5) .. (1.1,0.55);
\end{scope}

\draw[dashed] (0,0) -- (0.6,0.7);

\begin{scope}[scale=0.9]
\draw[fill=white] (1.1,0.55) .. controls (1.4,0.7) and (1.2,1) .. (1.1,1.1) .. controls (1.0,1.2) and (0.8,1.25) .. (0.7,1.3) .. controls (0.5,1.4) and (0.33,1.32) .. (0.3,1.2) .. controls (0.23,0.92) and (0.5,0.85) .. (0.6,0.7) .. controls (0.7,0.55) and (1,0.5) .. (1.1,0.55);
\end{scope}

\begin{scope}[scale=1.8]
\draw[fill=white] (1.1,0.55) .. controls (1.4,0.7) and (1.2,1) .. (1.1,1.1) .. controls (1.0,1.2) and (0.8,1.25) .. (0.7,1.3) .. controls (0.5,1.4) and (0.33,1.32) .. (0.3,1.2) .. controls (0.23,0.92) and (0.5,0.85) .. (0.6,0.7) .. controls (0.7,0.55) and (1,0.5) .. (1.1,0.55);
\end{scope}

\draw (1.5,1.75) -- (0.6,0.7);

\begin{scope}[scale=0.9]

\clip (1.1,0.55) .. controls (1.4,0.7) and (1.2,1) .. (1.1,1.1) .. controls (1.0,1.2) and (0.8,1.25) .. (0.7,1.3) .. controls (0.5,1.4) and (0.33,1.32) .. (0.3,1.2) .. controls (0.23,0.92) and (0.5,0.85) .. (0.6,0.7) .. controls (0.7,0.55) and (1,0.5) .. (1.1,0.55);

\draw (0,1.7) arc(90:0:0.54 and 1.95);
\draw (0,1.7) arc(90:0:0.795 and 1.92);
\draw (0,1.7) arc(90:0:1.03 and 1.85);
\draw (0,1.7) arc(90:0:1.23 and 1.787);
\draw (0,1.7) arc(90:0:1.4 and 1.71);

\draw (-0.15,1.14) arc(-90:0:1.7*0.565 and 0.6*0.565);
\draw (-0.15,0.99) arc(-90:0:1.7*0.66 and 0.6*0.66);
\draw (-0.15,0.82) arc(-90:0:1.7*0.745 and 0.6*0.745);
\draw (-0.15,0.645) arc(-90:0:1.7*0.82 and 0.6*0.82);
\draw (-0.15,0.46) arc(-90:0:1.7*0.89 and 0.6*0.89);
\end{scope}
\end{scope}
\end{scope}
\path (0,-3) node {$d=3$};
\end{tikzpicture}
\caption{Cases of $d=2$ and $d=3$}
\end{figure}


In this article we study the regularity theory of  the parabolic equation 
\begin{equation}\label{heat eqn}
u_t=\sum_{i,j=1}^d a^{ij}(t)u_{x^ix^j}+f=:\cL u+f, \quad t>0, x\in \cD
\; ; \; u(0,\cdot)=0
\end{equation}
under the Dirichlet boundary condition, where the diffusion coefficients $a_{ij}$  are real valued measurable functions of $t$, $a_{ij}=a_{ji}$, $i,j=1,\ldots,d$, and satisfy the uniform parabolicity condition, i.e. there exist  constants $\nu_1, \nu_2>0$ such that for any $t\in\mathbb{R}$ and $\xi=(\xi^1,\ldots,\xi^d)\in\mathbb{R}^d$,
\begin{eqnarray}
\nu_1 |\xi|^2\le \sum_{i,j}a_{ij}(t)\xi^i\xi^j\le \nu_2 |\xi|^2.  \label{uniform parabolicity}
\end{eqnarray}

Now we specify our condition on $\cM$. Since  $\overline{\cM}^S\neq S^{d-1}$, upon an appropriate rotation  we may assume 
$$s_0:=(0,0,\cdots,0,-1) \notin \overline{\cM}^S.
$$
 Thus we  can define the stereographic projection $\phi$ that maps the points of $S^{d-1}\setminus \{s_0\}$  onto  the tangent plane at $-s_0$ which we identify with $\bR^{d-1}$:
$$
\phi:S^{d-1}\setminus \{s_0\}\rightarrow \bR^{d-1}, \quad \phi(x',x^d)=\frac{2}{1+x^d}x'
$$
for $(x^1,\ldots,x^{d-1},x^d)=:(x',x^d)\in S^{d-1}\setminus \{s_0\}$. 
\begin{assumption}
\label{ass domain}
The set $\cM$ in $S^{d-1}$ is of class $\cC^2$, meaning that $\phi(\cM)$, the image of $\cM$ under $\phi$, has $\cC^2$ boundary in $\bR^{d-1}$.
\end{assumption}

To explain our main result in the frame of weighted Sobolev regularity, we  introduce appropriate function spaces.

Recall 
$$
\rho_{\circ}(x):=|x|,\quad \quad \rho(x)=\rho_{\cD}(x):=d(x,\partial\cD),
$$
which denote the distances from a point $x\in\cD$ to the vertex $0$ and to the boundary of $\cD$, respectively.
For $p\in(1,\infty)$, $\theta\in\bR$ and $\Theta\in \bR$,  define
$$
L_{p,\theta,\Theta}(\cD):=L_p\left(\cD,\rho_{\circ}^{\theta-\Theta}\rho^{\Theta-d}dx;\bR\right).
$$
That is, $L_{p,\theta,\Theta}(\cD)$ is  the class of real-valued functions $f$ such that 
$$
\|f\|^p_{L_{p,\theta,\Theta}(\cD)}:=\int_{\cD} |f|^p \rho_{\circ}^{\theta-\Theta}\rho^{\Theta-d}\,dx <\infty.
$$
Note that $\rho_{\circ}^{\theta-\Theta}\rho^{\Theta-d}=\rho_{\circ}^{\theta-d} \left(\frac{\rho}{\rho_{\circ}}\right)^{\Theta-d}$, which implies that our weight  captures the dependence of functions   on    $\rho_{\circ}$ and the ratio $\frac{\rho}{\rho_{\circ}}$. With this building block, for $n\in\{0,1,2,\ldots\}$  we define the function spaces
\begin{equation*}\label{space K norm}
K^n_{p,\theta,\Theta}(\cD)=\left\{f\,:\,\|f\|_{K^n_{p,\theta,\Theta}(\cD)}:=\sum_{|\alpha|\leq n}\|\rho^{|\alpha|}D^{\alpha}f\|_{L_{p,\theta,\Theta}(\cD)}<\infty\right\}.
\end{equation*}
Note  $K^0_{p,\theta,\Theta}(\cD)=L_{p,\theta,\Theta}(\cD)$, and for any integer $n\in\{0,1,2,\ldots\}$
\begin{equation}
   \label{eqn 4.20.1}
\|f\|_{K^n_{p,\theta+np,\Theta+np}(\cD)} = \sum_{|\alpha|\leq n} \|\rho^{|\alpha|+n}D^{\alpha}f\|_{L_{p,\theta,\Theta}(\cD)}.
\end{equation}

Below we list some  basic properties of the spaces $K^n_{p,\theta,\Theta}(\cD)$. More properties are discussed in Section 4. 

\begin{lemma}\label{property1}
Let $p\in(1,\infty)$, $\theta\in\bR$, $\Theta\in\bR$ and $n\in\{0,1,2,\ldots\}$.

\begin{enumerate}[align=right,label=\textup{(\roman*)}]
\item The space $K^n_{p,\theta,\Theta}(\cD)$ is a Banach space.

\item For any $\mu\in\bR$
\begin{align*}
N^{-1}\|f\|_{K^n_{p,\theta,\Theta}(\cD)}\leq\|\rho_{\circ}^{\mu}f\|_{K^n_{p,\theta-\mu p,\Theta}(\cD)}\leq N\|f\|_{K^n_{p,\theta,\Theta}(\cD)},
\end{align*} 
where $N=N(d,n,p,\mu)$.
\item If $n\geq 1$,  then the differential operator $D_i:K^n_{p,\theta,\Theta}(\cD)\to K^{n-1}_{p,\theta+p,\Theta+p}(\cD)$ is  bounded for any $i=1,\ldots,d$. Moreover,  we have
\begin{equation*}
                  \label{eqn 4.16.1}
\|D^{\alpha}f\|_{K^{n-|\alpha|}_{p,\theta+|\alpha|p,\Theta+|\alpha|p}(\cD)}\leq  \|f\|_{K^n_{p,\theta,\Theta}(\cD)}
\end{equation*}
for any multi-index $\alpha$ satisfying $|\alpha|\le n$. 
   
    \item  Let $R>1$ and $f\in L_{1,loc}(\cD)$.
If
$$
supp(f)\subset V_R:=\left\{x\in\cD\,:\,\frac{1}{R}<\rho(x)<R\right\},
$$ then $f\in K^n_{p,\Theta,\Theta}(\cD)$ if and only if $f\in W^n_{p}(\cD)$. Moreover
\begin{align*}
N^{-1}\|f\|_{K^n_{p,\Theta,\Theta}(\cD)}\leq \|f\|_{W^n_p(\cD)}\leq N\|f\|_{K^n_{p,\Theta,\Theta}(\cD)},
\end{align*}
where  $N=N(d,n,p,\Theta,R)$.

\item $\cC_c^{\infty}(\cD)$ is dense in $K^n_{p,\theta,\Theta}(\cD)$.

\end{enumerate}
\end{lemma}

\begin{proof}
The proof of (i) is  straightforward and left to the reader. (ii) is due to the observation that
\begin{align}\label{estimate.rho}
\sup_{x\in\cD}\left( |\rho_{\circ}^{|\alpha|-\mu}(x)(D^{\alpha}\rho_{\circ}^{\mu})(x)| \right)<\infty
\end{align}
holds  for any $\mu\in\bR$ and multi-index $\alpha$.
 (iii) follows the definition of the norm. (iv) is obvious  since $\rho_{\circ}$ is bounded below and $\rho$ is bounded from above and below by positive constants on the support of $f$.

Let us prove (v). First note that by (ii), without loss of generality we  may assume that $\theta=\Theta$.
Let $f\in K^n_{p,\theta,\Theta}(\cD)$. Thus,  
\begin{align*}
\|f\|^p_{K^n_{p,\Theta,\Theta}(\cD)}=\sum_{|\alpha|\leq n}\int_{\cD}|\rho^{|\alpha|}D^{\alpha}f|^p \rho^{\Theta-d}dx<\infty.
\end{align*}
We choose a sequence of  infinitely differentiable functions $\xi_m$ such that $\xi_m$ has support in $V_{2m}$, $0\leq \xi_m\leq 1$, $\xi_m(x)\to 1$ as $m\to \infty$ for $x\in \cD$, and  $\rho^{|\beta|}D^{\beta}\xi_m$ is uniformly bounded and goes to zero as $m\to \infty$ for any mullti-index $\beta$ with $|\beta|\geq 1$. For instance, one can construct such functions as follows.  Choose  a nonnegative function $\zeta \in \cC_c^{\infty}(\bR^d)$ satisfying $supp\,(\zeta)\subset B_1(0)$ and $\int_{\bR^d}\zeta(x) dx=1$, and let
\begin{align*}
\zeta^{(\epsilon)}=\frac{1}{\epsilon^d}\zeta\left(\frac{\cdot}{\epsilon}\right)
\end{align*}
for any $\epsilon>0$.
For $m\in\bN$, let us define
\begin{align*}
\xi_m:=1_{\cD_{1/m}}*\zeta^{(\frac{1}{2m})}-1_{\cD_m}*\zeta^{(\frac{m}{2})},
\end{align*}
where
\begin{align*}
\cD_r=\left\{x\in\cD\,:\,\rho(x)>r\right\}, \quad r>0,
\end{align*}
and $*$ denotes the convolution of two functions involved.
Then, as we intended, $\xi_m$ satisfies
$$
supp(\xi_m) \subset V_{2m}
$$
and
\begin{align*}
|D^{\beta}\xi_m(x)|\leq N \rho^{-|\beta|}(x)1_{V_{2m}\setminus V_{\frac{m}{2}}}(x),\ \ \ \forall x\in \bR^d
\end{align*}
for any multi-index $\beta$ with $|\beta|\geq 1$,
where $N$ is independent of $x$ and $m$.

Note that
$$
\lim_{m\rightarrow\infty}\left(1-1_{V_{2m}}\right)=\lim_{m\rightarrow\infty}\left(1_{V_{2m}\setminus V_{\frac{m}{2}}}\right)=0
$$
pointwise. Hence, we get
\begin{align*}
\lim_{m\rightarrow \infty}\|f-f\xi_m\|^p_{K^n_{p,\Theta,\Theta}}=0
\end{align*}
by Lebesgue's dominating convergence theorem with a dominating function
$$
\sum_{|\alpha|\leq n}|\rho^{|\alpha|}D^{\alpha}f|^p \rho^{\Theta-d}.
$$
Since $supp\,(f\xi_m)\subset V_{2m}$, $f\xi_m$ is in $W^n_p(\cD)$. For each $m$, by mollifying and cutting off,  we choose $g_{m,k}\in\cC_c^{\infty}(V_{4m})$ such that $g_{m,k}\rightarrow f \xi_m$  as $k\rightarrow \infty$ in $W^n_p(\cD)$ and hence in $K^n_{p,\Theta,\Theta}(\cD)$ by (iv), meaning that we can choose
 $g_m\in \cC_c^{\infty}(\cD)$ satisfying $$\|f\xi_m-g_m\|_{K^n_{p,\Theta,\Theta}(\cD)}\leq 2^{-m}.$$
Consequently, we get
\begin{align*}
\limsup_{m\rightarrow \infty}\|f-g_m\|_{K^n_{p,\Theta,\Theta}(\cD)}\leq \limsup_{m\rightarrow \infty}\left(\|f-f\xi_m\|_{K^n_{p,\Theta,\Theta}(\cD)}+2^{-m}\right)=0,
\end{align*}
and  (v) is proved. 
\end{proof}

Finally we  introduce our function space in which  the solution $u$ to equation \eqref{heat eqn} lies.  For    $T\in (0,\infty)$,  $p\in (1,\infty)$, $\theta\in\bR$, $\Theta\in\bR$, and $n\in\{0,1,2,\ldots\}$, we define the function spaces
\begin{equation*}\label{space time K norm}
\bK^n_{p,\theta,\Theta}(\cD,T):=L_p\left((0,T];K^n_{p,\theta,\Theta}(\cD)\right) 
\end{equation*}
with $\bL_p(\cD,T):=\bK^0_{p,\theta,\Theta}(\cD,T)$.

\begin{remark}
\label{remark dense}
By modifying the proof of Lemma \ref{property1}(v), based on a mollification with respect to both time and space variables,  one can prove that
$\cC^{\infty}_c\left((0,T)\times\cD\right)$ is dense in $\bK^n_{p,\theta,\Theta}(\cD,T)$.
\end{remark}

Now we define our sense of solution together with the space for the source $f$.

\begin{defn}\label{first spaces}
Let $p\in(1,\infty)$, $\theta\in\bR$, $\Theta\in\bR$ and $n\in\{0,1,2,\ldots\}$.

(i)  We write $u\in\cK^{n+2}_{p,\theta,\Theta}(\cD,T)$ if
 $u\in \bK^{n+2}_{p,\theta-p,\Theta-p}(\mathcal{D},T)$ and there exists $\tilde{f}\in\bK^n_{p,\theta+p,\Theta+p}(\mathcal{D},T)$ such that 
 \begin{align*}
u_t=\tilde{f},\quad t\in(0,T]\quad ; \quad u(0,\cdot)=0
\end{align*}
in the sense of distributions on $\cD$, that is,  for any $\varphi\in \cC_c^{\infty}(\cD)$  the equality
\begin{equation*}
  \label{eqn sol}
(u(t,\cdot),\varphi)=\int^{t}_{0}(\tilde{f}(s,\cdot),\varphi)ds
\end{equation*}
holds for all $t\in (0,T]$.
The norm in $\cK^{n+2}_{p,\theta,\Theta}(\cD,T)$ is defined by
\begin{align*}
\|u\|_{\cK^{n+2}_{p,\theta,\Theta}(\cD,T)}:=\|u\|_{\bK^{n+2}_{p,\theta-p,\Theta-p}(\cD,T)}+\|u_t\|_{\bK^n_{p,\theta+p,\Theta+p}(\cD,T)}.
\end{align*}

(ii) We say that $u$ is a solution to equation (\ref{heat eqn}) in  $\cK^{n+2}_{p,\theta,\Theta}(\mathcal{D},T)$ if  the source  $f$ is in $\bK^n_{p,\theta+p,\Theta+p}(\mathcal{D},T)$ and $u\in \bK^{n+2}_{p,\theta-p,\Theta-p}(\mathcal{D},T)$ satisfies 
\begin{align*}
u_t=\cL u+f, \quad t\in(0,T] \quad ; \quad u(0,\cdot)=0
\end{align*}
in the sense of distributions on $\cD$. 
\end{defn}
\vspace{0.1cm}
\begin{remark}\label{space time K space remark}
By Lemma \ref{property1} (iii), if $u\in\bK^{n+2}_{p,\theta-p,\Theta-p}(\mathcal{D},T)$, then $\cL u \in \bK^n_{p,\theta+p,\Theta+p}(\cD,T)$. This supports Definition \ref{first spaces} (ii).
\end{remark}

\begin{thm}
For $p\in(1,\infty)$, $\theta\in\bR$, $\Theta\in\bR$, and $n\in\{0,1,2,\ldots\}$, $\cK^{n+2}_{p,\theta,\Theta}(\cD,T)$ is a Banach space.
\end{thm}

\begin{proof}
The completeness of the space $\cK^{n+2}_{p,\theta,\Theta}(\cD,T)$ can be proved by repeating  the argument in Remark 3.8 of \cite{Krylov 2001}, in which the completeness is proved for the special case of $\cD=\bR^d_+$ and $\theta=\Theta$, meaning that only the distance to the boundary is involved in the weight, nevertheless the argument  works for us.
\end{proof}

\begin{remark}
In Definition \ref{first spaces}  the ranges of $\theta$ and $\Theta$ are still open. However, there is no guarantee  yet that there is a solution in $\bK^{n+2}_{p,\theta-p,\Theta-p}(\mathcal{D},T)$ for arbitrary   $p\in (1,\infty)$, $\theta\in\bR$,  $\Theta\in\bR$, and $f\in  \bK^n_{p,\theta+p,\Theta+p}(\mathcal{D},T)$. Particularly, since we are assuming  \emph{zero Dirichlet boundary condition}, the source function $f$ and $\theta,\Theta$ are needed to be appropriately chosen.

It turns out that the adimissible range of $\theta$ for $f$ (and hence for $u$) is affected by \emph{the shape of the conic domain} $\cD=\cD(\cM)$, the uniform parabolicity of the diffusion coefficients, the space dimension $d$, and the summability parameter $p$, while $\Theta$ depends only on $d$ and  $p$.
\end{remark}

To explain the admissible range of $\theta$ for equation \eqref{heat eqn} we need the following definition.

\begin{defn}[see Section 2 of \cite{Kozlov Nazarov 2014}]\label{lambda}

 One can refer to  Section \ref{sec:Introduction} for some of the notations below.

(i) By  $\lambda^+_{c,\cL}$ we denote the supremum  of all  $\lambda \geq 0$ such that for some constant $K_0=K_0(\cL,\cM,\lambda)$  it holds that 
\begin{equation}
      \label{eqn 8.12.1}
|u(t,x)|\le K_0 \left(\frac{|x|}{R}\right)^{\lambda}\sup_{Q^{\mathcal{D}}_{\frac{3R}{4}}(t_0,0)}\ |u|,
\quad
\forall \;(t,x)\in Q^{\mathcal{D}}_{R/2}(t_0,0)
\end{equation}
for any $R>0$, $t_0$ and $u$ belonging to  $\mathcal{V}_{loc}(Q^{\mathcal{D}}_R(t_0,0))$  and satisfying
\begin{equation*}
u_t=\mathcal{L}u\quad \text{in}\; Q^{\mathcal{D}}_R(t_0,0)\quad ; \;\quad
u(t,x)=0\quad\text{for}\;\; x\in\partial\mathcal{D}.
\end{equation*}

(ii)  By $\lambda^-_{c,\cL}$ we denote supremum of all $\lambda\geq 0$ with above property for the operator 
$$\hat{\mathcal{L}}=\sum_{i,j}a_{ij}(-t)D_{ij}.
$$
\end{defn}

Although we consider one fixed operator $\cL$ in this article, we want to pose the following definition as a preparation for our subsequent article  on stochastic parabolic equations, for which, as we mentioned in the introduction, the result of this article will serve crucially. In the case of stochastic parabolic equations, the operator will be random and involve infinitely many  operators.  The definition is used to establish explicit dependency of constants appearing in our estimates.

\begin{defn}
\label{def 8.12}
(i) Let $\cT_{\nu_1,\nu_2}$ denote collection of all operators $\tilde{\cL}=\sum_{i,j=1}^d \tilde{a}^{ij}(t)D_{ij}$ such that $\tilde{A}(t):=(\tilde{a}^{ij}(t))_{d\times d}$ is measurable in $t$
and satisfies the uniform parabolicity condition \eqref{uniform parabolicity}.

(ii) By  $\lambda_{c}(\nu_1,\nu_2)$ we denote the supremum  of all  $\lambda \geq 0$ such that for some constant $K_0=K_0(\nu_1,\nu_2,\cM,\lambda)$  it holds that
for any $\tilde{\cL}\in \cT_{\nu_1,\nu_2}$, $R>0$, $t_0$,
\begin{equation}
      \label{eqn 8.12.111}
|u(t,x)|\le K_0 \left(\frac{|x|}{R}\right)^{\lambda}\sup_{Q^{\mathcal{D}}_{\frac{3R}{4}}(t_0,0)}\ |u|,
\quad
\forall \;(t,x)\in Q^{\mathcal{D}}_{R/2}(t_0,0),
\end{equation}
provided that $u$ belongs to  $\mathcal{V}_{loc}(Q^{\mathcal{D}}_R(t_0,0))$  and satisfies
\begin{equation*}
u_t=\tilde{\cL} u\quad \text{in}\; Q^{\mathcal{D}}_R(t_0,0)\quad ; \;\quad
u(t,x)=0\quad\text{for}\;\; x\in\partial\mathcal{D}.
\end{equation*}
\end{defn}

\begin{remark}
(i) Note that the dependency of $K_0$ in Definition \ref{def 8.12} is more explicit compared to that of Definition \ref{lambda}.  By definitions, we have $\lambda^{\pm}_{c,\cL}\geq \lambda_c(\nu_1,\nu_2)$ if $\cL\in \cT_{\nu_1,\nu_2}$. 

(ii) The values of $\lambda^{\pm}_{c,\cL}$ and $\lambda_c(\nu_1,\nu_2)$ do not change if one replaces $\frac{3}{4}$  in \eqref{eqn 8.12.1} and  \eqref{eqn 8.12.111} by any number in $(1/2,1)$ 
(see \cite[Lemma 2.2]{Kozlov Nazarov 2014}). 
\end{remark}

Both  $\lambda^{+}_{c,\cL}$  and  $\lambda^-_{c,\cL}$  definitely depend on $\mathcal{M}$ and $\cL$.   Below are some  sharp estimates of $\lambda^+_{c,\cL}$ and 
$\lambda^-_{c,\cL}$. See \cite{Green, Kozlov Nazarov 2014} for more informations.

\begin{prop}
\label{prop Theta}

\begin{enumerate}[align=right,label=\textup{(\roman*)}]
\item  If $\mathcal{L}=\Delta_x$,  then
\begin{equation*}\label{CUB1}
\lambda^{\pm}_{c,\cL}=-\frac{d-2}{2}+\sqrt{\Lambda+\frac{(d-2)^2}{4}}  >0,
\end{equation*}
where $\Lambda$ is the first eigenvalue of Laplace-Beltrami operator with the Dirichlet condition on  $\mathcal{M}$.
In particular, if $d=2$ and $\cD=\cD^{(\kappa)}$ (see (\textsl{\ref{wedge in 2d}})), then
\begin{equation*}
   \label{kappa-1}
\lambda^{\pm}_{c,\cL}=\frac{\pi}{\kappa}.
\end{equation*}

\item Let $0<\nu_1<\nu_2$. Then for any $\tilde{\cL}\in \cT_{\nu_1,\nu_2}$,
\begin{equation*}\label{CUB3}
\lambda^{\pm}_{c,\tilde{\cL}}\geq \lambda_{c}(\nu_1,\nu_2) \geq \left(-\frac{d-2}{2}+\sqrt{\frac{\nu_1}{\nu_2}}\sqrt{\Lambda+\frac{(d-2)^2}{4}}\right) \vee 0.
\end{equation*}
\end{enumerate}
\end{prop}
\begin{proof}
See \cite[Theorem 2.4]{Kozlov Nazarov 2014} (i) and \cite[Theorem 3.2]{Green} for (ii).
\end{proof}

Here is the main result of this article. The proof  is placed in Section \ref{sec:main result}.

\begin{thm}\label{main result}
Let Assumption \ref{ass domain} and condition (\textsl{\ref{uniform parabolicity}}) hold, 
$p\in(1,\infty)$, and $n\in\{0,1,2,\ldots\}$. Also, assume that    $\Theta\in\bR$ and $\theta\in\bR$ satisfy
\begin{equation*}
d-1<\Theta<d-1+p  \quad\text{and}\quad   \label{theta}
p(1-\lambda^+_{c,\cL})<\theta<p(d-1+\lambda^-_{c,\cL}).
\end{equation*}
Then for any $f\in \bK^n_{p,\theta+p,\Theta+p}(\cD,T)$  there exists a unique solution $u$ in $\cK^{n+2}_{p,\theta,\Theta}(\cD,T)$ to  equation (\textsl{\ref{heat eqn}}). 
Moreover, the inequality
\begin{equation}
\label{main estimate}
\|u\|_{\cK^{n+2}_{p,\theta,\Theta}(\cD,T)}\leq N\,\|f\|_{\bK^n_{p,\theta+p,\Theta+p}(\cD,T)}
\end{equation}
holds with a constant $N=N(\cM,d,p,n,\theta,\Theta, \cL)$. Moreover, if
\begin{equation}
   \label{eqn 8.12.7}
p\big(1-\lambda_c(\nu_1,\nu_2)\big)<\theta<p\big(d-1+\lambda_{c}(\nu_1,\nu_2)\big),
\end{equation}
then the constant $N$ in (\textsl{\ref{main estimate}}) depends only on $\cM,d,p,n,\theta,\Theta,\nu_1$ and $\nu_2$.
\end{thm}

\begin{remark}
This is a good place to explain why  the solution $u$ in Theorem \ref{main result} satisfied zero Dirichlet boundary condition.  Under the assumption $d-1<\Theta<d-1+p$, \cite[Theorem 2.8]{doyoon} implies that the trace operator  is well defined for functions $u\in \bK^2_{p,\theta-p, \Theta-p}(\cD,T)$, and hence by Lemma \ref{property1} (v)  we have $u|_{\partial \cD}=0$.
\end{remark}

\begin{remark} 
Due to  \eqref{eqn 4.20.1}, estimate \eqref{main estimate} certainly yields
\begin{eqnarray}
\nonumber
&&\sum_{|\alpha|\leq n+2} \int^T_0 \int_{\cD} |\rho^{|\alpha|-1}D^{\alpha}u|^p \rho_{\circ}^{\theta-\Theta}\rho^{\Theta-d}\, dxdt \\
&& \quad \quad \leq N 
\sum_{|\alpha|\leq n}  \int^T_0 \int_{\cD} |\rho^{|\alpha|+1}D^{\alpha}f|^p \rho_{\circ}^{\theta-\Theta}\rho^{\Theta-d}\, dxdt.
 \label{main estimate-1}
\end{eqnarray}

\end{remark}

\begin{remark}
\label{remark 5.1.5}
If $d=2$, $\cD=\cD^{(\kappa)}$ of \eqref{wedge in 2d}, and $\cL=\Delta_x$, then the condition $p(1-\lambda^+_{c,\cL})<\theta<p(d-1+\lambda^-_{c,\cL})$ becomes
\begin{equation}
 \label{con Theta}
p(1-\frac{\pi}{\kappa})<\theta<p(1+\frac{\pi}{\kappa}).
\end{equation}

If $\kappa=\pi$, then $\cD$ is a half space in $\bR^2$. In this case the admissible range of $\theta$ is $(0, 2p)$ which surely contains the range $(1,p+1)$ of $\Theta$. Hence, we are safe to take $\theta=\Theta$ in \eqref{main estimate-1}  and  get
\begin{equation*}
 \label{eqn 4.21.1}
\int^T_0 \int_{\cD} \left(|\rho^{-1}u|^p +|u_x|^p + |\rho u_{xx}|^p\right) \rho^{\Theta-d}\, dx\,dt  
 \leq N \int^T_0 \int_{\cD} |\rho f|^p \rho^{\Theta-d}\, dx\,dt,
\end{equation*}
for any $\Theta\in (1,p+1)$. This fits into the result of \cite{Krylov 1999-1}, and thus our result extends the main result in \cite{Krylov 1999-1} up to the conic domains at least in two-dimensional space provided that 
\begin{equation}
\label{eqn 6.30.1}
\Theta\in  \left(p(1-\frac{\pi}{\kappa}), \, p(1+\frac{\pi}{\kappa})\right).
\end{equation}
 One can notice that, for any  fixed $\Theta\in (1,1+p)$,   \eqref{eqn 6.30.1} holds  for all $p>1$ if $\kappa\leq \pi$, and  if $\kappa>\pi$ then  \eqref{con Theta}  holds only for sufficiently small $p$. The bigger the angle $\kappa$ is, the less the summability of derivatives is.

\end{remark}

\begin{remark}
Theorem 2.1 in \cite{Kozlov Nazarov 2014} gives an $L_{p,q}$-estimate with the weight system involving only the distance to the vertex with the range of $\mu=\frac{\theta-d}{p}+1$ given by
\begin{equation}\label{range of mu}
2-\frac{d}{p}-\lambda^+_{c,\cL}<\mu<d-\frac{d}{p}+\lambda^-_{c,\cL}.
\end{equation}
\eqref{range of mu} is the same as $p(1-\lambda^+_{c,\cL})<\theta<p(d-1+\lambda^-_{c,\cL})$ and  the result  with $p=q$ there  fits into \eqref{main estimate-1} with $\Theta=d$ and $n=0$ since $\rho\le\rho_{\circ}$. 

\end{remark}

\mysection{Key Estimate}

In this section we prove Lemma \ref{main est} below, which plays the key role when we prove our main result, Theorem \ref{main result}, in Section \ref{sec:main result}.

 Let $G(t,s,x,y)$ denote the Green's function for the operator $\partial_t-\cL$ with the domain $\cD=\cD(\cM)$. By definition,    $G$ is nonnegative and, for any fixed $s\in \mathbb{R}$ and $y\in\mathcal{D}$,  the function $v=G(\ \cdot,s,\ \cdot,y)$ satisfies
$$
v_t=\mathcal{L}v\quad\textrm{in}\quad (s,\infty)\times\mathcal{D}\; ; \quad v=0\quad \textrm{on}\quad (s,\infty)\times\mathcal{\partial D} \; ; \quad  v(t,\cdot)=0\quad \textrm{for} \quad t<s.
$$

\vspace{2mm}

Here  is the main result of this section.

\begin{lemma}\label{main est}
Let $p\in(1,\infty)$, and let $\theta\in\bR$ and $\Theta\in\bR$ satisfy $$p(1-\lambda^+_{c,\cL})<\theta<p(d-1+\lambda^-_{c,\cL})\quad\text{and}\quad d-1<\Theta<d-1+p.
$$
 Then for any  $f\in \bL_{p,\theta+p,\Theta+p}(\cD,T)$ and  the function $u$ defined by
$$
u(t,x):=\int_0^t\int_{\cD}G(t,s,x,y)f(s,y)dyds\,,
$$
u is in $\bL_{p,\theta-p,\Theta-p}(\cD,T)$, and the estimate
\begin{align}
\int^T_0\int_{\cD}|\rho^{-1}u|^p\rho_{\circ}^{\theta-\Theta}\rho^{\Theta-d}dxdt\leq N\int^T_0\int_{\cD}|\rho\,f|^p\rho_{\circ}^{\theta-\Theta}\rho^{\Theta-d}dxdt
\label{main inequality}
\end{align}
holds with  $N=N(\cM,d,p,\theta,\Theta,\cL)$. Moreover, if 
$$
p\big(1-\lambda_c(\nu_1,\nu_2)\big)<\theta<p\big(d-1+\lambda_{c}(\nu_1,\nu_2)\big),
$$
then the constant $N$ in (\textsl{\ref{main inequality}}) depends only on $\cM,d,p,\theta,\Theta,\nu_1$ and $\nu_2$.
\end{lemma}

To prove Lemma~\ref{main est}, we need two quantitative lemmas below, Lemma  ~\ref{lemma3.1} and  Lemma  ~\ref{lemma3.2S}.

\begin{lemma}\label{lemma3.1}
Let $\alpha+\beta>0,\ \beta > 0$, and $\gamma>0$. Then   there exists a constant  $N(\alpha,\beta,\gamma)>0$ such that
\begin{align}\label{eq.lemma3.1}
\int^{\infty}_0 \frac{1}{\left(\sqrt{t}+a\right)^{\alpha}\left(\sqrt{t}+b\right)^{\beta+\gamma}t^{1-\frac{\gamma}{2}}}dt \leq \frac{N}{a^\alpha b^\beta}
\end{align}
for any $a\geq b>0$.
\end{lemma}
\begin{proof}

Multiplying both sides of \eqref{eq.lemma3.1} by $a^{\alpha}b^{\beta}$, we see that is is enough to prove  
$$
\int^{\infty}_0 \left(\frac{a}{\sqrt{t}+a}\right)^{\alpha}\left(\frac{b}{\sqrt{t}+b}\right)^{\beta}\left(\frac{\sqrt{t}}{\sqrt{t}+b}\right)^{\gamma}\frac{dt}{t}\,
$$
is bounded by a constant $N=N(\alpha,\beta,\gamma)$.

- {\bf{Case 1.}} $\alpha\geq 0$.

Since $a>0$ and $\alpha\geq0$, we have
\begin{align*}
\left(\frac{a}{\sqrt{t}+a}\right)^{\alpha} \leq  1.
\end{align*}
Hence, we get
\begin{align*}
\int^{\infty}_0 \left(\frac{a}{\sqrt{t}+a}\right)^{\alpha}\left(\frac{b}{\sqrt{t}+b}\right)^{\beta}\left(\frac{\sqrt{t}}{\sqrt{t}+b}\right)^{\gamma}\frac{dt}{t}
&\leq \int^{\infty}_0 \left(\frac{b}{\sqrt{t}+b}\right)^{\beta}\left(\frac{\sqrt{t}}{\sqrt{t}+b}\right)^{\gamma}\frac{dt}{t}\\
&=\int^{\infty}_0 \frac{1}{\left(\sqrt{s}+1\right)^{\beta+\gamma} s^{1-\frac{\gamma}{2}}} \,ds,
\end{align*}
where the last quantity follows the change of variable, $s=t/b^2$ and it is finite since 
$1-\frac{\gamma}{2}<1$ and  $1+\frac{\beta}{2}>1$.

- {\bf{Case 2.}} $\alpha< 0$.

Since $\alpha< 0$ and $a\geq b>0$,
we have
\begin{align*}
\left(\frac{a}{\sqrt{t}+a}\right)^{\alpha}\leq \left(\frac{b}{\sqrt{t}+b}\right)^{\alpha}.
\end{align*}
Hence,
\begin{eqnarray*}
&&\int^{\infty}_0\left(\frac{a}{\sqrt{t}+a}\right)^{\alpha}\left(\frac{b}{\sqrt{t}+b}\right)^{\beta}\left(\frac{\sqrt{t}}{\sqrt{t}+b}\right)^{\gamma}\frac{1}{t}\,dt\\
&\leq& \int^{\infty}_0\left(\frac{b}{\sqrt{t}+b}\right)^{\alpha+\beta}\left(\frac{\sqrt{t}}{\sqrt{t}+b}\right)^{\gamma}\frac{1}{t}\,dt\\
&\leq& N(\alpha,\beta,\gamma)
\end{eqnarray*}
since $\alpha+\beta>0$, $\gamma>0$ and hence we can use the argument of Case 1.
\end{proof} 

The following lemma is a particular  result of Lemma ~\ref{lemma3.2S} with $\cD=\bR^d_+$ and  will be used in  in the proof of Lemma ~\ref{lemma3.2S}.
\begin{lemma}\label{lemma3.2R}
Let $\sigma>0,\ \alpha+\gamma>-d$, $\gamma>-1$ and $\beta,\ \omega \in \mathbb{R}$.
Then there exists a constant $N(d, \alpha, \beta, \gamma,\omega,\sigma)>0$ such that 
\begin{align}\label{eq.lemma3.2R}
\int_{\mathbb{R}^d} \frac{|y|^{\alpha}}{\left(|y|+1\right)^{\beta}}\frac{|y^1|^{\gamma}}{\left(|y^1|+1\right)^{\omega}}\ e^{-\sigma |x-y|^2} dy \leq N \left(|x|+1\right)^{\alpha-\beta}\left(|x^1|+1\right)^{\gamma-\omega}
\end{align}
for any $x=(x^1,\ldots,x^d)\in\bR^d$.
\end{lemma}

\begin{proof} 
{\bf 1}.  Moving $(|x|+1)^{\alpha-\beta}(|x^1|+1)^{\gamma-\omega}$ to the left hand side of \eqref{eq.lemma3.2R}, and then using the change of variables $x-y\to y$, we note that it is enough to show that 
\begin{align*}
I(x):&=\int_{\mathbb{R}^d} \frac{1}{\left(|x|+1\right)^{\alpha-\beta}}\frac{1}{\left(|x^1|+1\right)^{\gamma-\omega}}\frac{|y|^{\alpha}}{\left(|y|+1\right)^{\beta}}\frac{|y^1|^{\gamma}}{\left(|y^1|+1\right)^{\omega}}\,e^{-\sigma |x-y|^2}dy\\
&=\int_{\mathbb{R}^d} \left(\frac{|x-y|}{|x-y|+1}\right)^{\alpha}\left(\frac{|x^1-y^1|}{|x^1-y^1|+1}\right)^{\gamma} f(x,y) \,dy
\end{align*}
is bounded by a constant  $N=N(d, \alpha, \beta, \gamma,\omega,\sigma)$, where
\begin{align*}
f(x,y)=\left(\frac{|x-y|+1}{|x|+1}\right)^{\alpha-\beta}\left(\frac{|x^1-y^1|+1}{|x^1|+1}\right)^{\gamma-\omega}e^{-\sigma|y|^2}.
\end{align*}
By the observation
\begin{align*}
\frac{|x-y|+1}{|x|+1}\leq \frac{|x|+|y|+1}{|x|+1}\leq |y|+1,\quad\frac{|x|+1}{|x-y|+1}\leq \frac{|x-y|+|y|+1}{|x-y|+1}\leq |y|+1,
\end{align*}
and the similar observation for $x^1$ and $y^1$, we get
\begin{align*}
f(x,y)&\leq \left(|y|+1\right)^{|\alpha-\beta|}\left(|y^1|+1\right)^{|\gamma-\omega|}e^{-\sigma|y|^2}\\
&\leq N
\Big\{\left(|y'|+1\right)^{|\alpha-\beta|}\left(|y^1|+1\right)^{|\gamma-\omega|}e^{-\sigma|y|^2}\\
&\qquad\qquad\qquad\qquad\qquad\qquad\quad +\left(|y^1|+1\right)^{|\alpha-\beta|+|\gamma-\omega|}e^{-\sigma|y|^2}\Big\}\\
&=N(d,\alpha,\beta)\left(\psi_1(y')\phi_1(y^1)+\psi_2(y')\phi_2(y^1)\right),
\end{align*}
where $y=(y^1,y')\in \bR^1\times \bR^{d-1}$ and
\begin{align*}
&\psi_1(y')=\left(|y'|+1\right)^{|\alpha-\beta|}e^{-\sigma|y'|^2},\quad\phi_1(y^1)=\left(|y^1|+1\right)^{|\gamma-\omega|}e^{-\sigma|y^1|^2}\,, \\
&\psi_2(y')=e^{-\sigma|y'|^2}\quad \text{and} \quad \phi_2(y^1)=\left(|y^1|+1\right)^{|\alpha-\beta|+|\gamma-\omega|}e^{-\sigma|y^1|^2}\,.
\end{align*}
Therefore, we have
\begin{align*}
I(x)\leq N
 \sum_{i=1}^2\,\int_{\mathbb{R}^d}\left(\frac{|x-y|}{|x-y|+1}\right)^{\alpha}\left(\frac{|x^1-y^1|}{|x^1-y^1|+1}\right)^{\gamma}\psi_i(y')\phi_i(y_1)dy.
\end{align*}

{\bf 2}. Noting that there exists a constant $C=C(d,\alpha,\beta,\gamma,\omega,\sigma)>0$ such that
\begin{align*}
||\psi_i||_{L_1(\mathbb{R}^{d-1})},\,\,\max_{\bR^{d-1}}|\psi_i|,\,\,||\phi_i||_{L_1(\mathbb{R})},\,\,\max_{\bR}|\phi_i|\leq C \quad \text{for}\ i=1,\ 2,
\end{align*}
we only need to show that there exists $N(d,\alpha,\gamma,C)$ such that
\begin{align*}
I'(x):=\int_{\mathbb{R}^d}\left(\frac{|x-y|}{|x-y|+1}\right)^{\alpha}\left(\frac{|x^1-y^1|}{|x^1-y^1|+1}\right)^{\gamma}\psi(y')\phi(y_1)dy\leq N\\
\end{align*}
for all $x$, provided that
$$||\psi||_{L_1(\mathbb{R}^{d-1})},\quad\max_{\bR^{d-1}}|\psi|,\quad ||\phi||_{L_1(\mathbb{R})},\quad\max_{\bR}|\phi|\;\;\leq C$$
for some constant $C>0$.

- {\bf{Case 1.}} $\alpha>-d+1$.

Put
\begin{align*}
I''(x,y)=\int_{\bR^{d-1}}\left(\frac{|x-y|}{|x-y|+1}\right)^{\alpha}\psi(y')dy'.
\end{align*}
If $\alpha\geq 0$, we instantly get
\begin{align*}
I''(x,y)\leq \int_{\bR^{d-1}}\psi(y')dy'\leq C.
\end{align*}
If $-d+1<\alpha<0$, we also have 
$$\left(\frac{|x-y|+1}{|x-y|}\right)^{-\alpha}\leq \left(1+\frac{1}{|x'-y'|}\right)^{-\alpha}\leq N(\alpha)\left(1+|x'-y'|^{\alpha}\right)$$
for a constant $N(\alpha)$.
Hence, we get
\begin{align*}
I''(x,y)&\leq N\left(\int_{\mathbb{R}^{d-1}}\psi(y')\ dy'+\int_{\mathbb{R}^{d-1}}|x'-y'|^{\alpha}\,\psi(y')\ dy'\right)\\
&\leq N\left(2\ ||\psi||_{L_1(\mathbb{R}^{d-1})}+\max_{\bR^{d-1}}|\psi|\int_{|x'-y'|<1}|x'-y'|^{\alpha}dy'\right)\\
&\leq N(d,\alpha, C),
\end{align*}
and, for all $\alpha>-d+1$, we have
$$
I'(x)\leq N(d,\alpha,C)\int_{\mathbb{R}}\left(\frac{|x^1-y^1|}{|x^1-y^1|+1}\right)^{\gamma}\phi(y^1)dy^1.
$$
Then, keeping the condition $\gamma>-1$ in mind and using the similar argument above, we have
$$
I'(x)\leq N(d,\alpha,\gamma,C)
$$
for all $x$.

- {\bf{Case 2.}} $\alpha\leq -d+1$.

Since $\alpha\leq -d+1$ and $\alpha+\gamma>-d$, we note
$$
\gamma+1>-\alpha-d+1\geq0.
$$
Choose any
$$
\delta\in\left(-\alpha-d+1,\ \gamma+1 \right)\subset\left(0,\ \infty\right).
$$
Since $\delta> 0$, we have
$$
\left(\frac{|x^1-y^1|}{|x^1-y^1|+1}\right)^{\delta}\leq \left(\frac{|x-y|}{|x-y|+1}\right)^{\delta}.
$$
Hence, we get
\begin{align*}
I'(x)\leq \int_{\mathbb{R}^d}\left(\frac{|x-y|}{|x-y|+1}\right)^{\alpha+\delta}\left(\frac{|x_1-y_1|}{|x_1-y_1|+1}\right)^{\gamma-\delta}\psi(y')\phi(y_1)dy\leq N(d,\alpha,\gamma,C)
\end{align*}
by $\alpha+\delta>-d+1$, $\gamma-\delta>-1$, and the argument of Case 1.
\end{proof}

In Lemma~\ref{lemma3.2R}, the first coordinate  $x^1$ plays the role of  the distance between $x\in\bR^d$ and $\partial\bR^d_+$.  For
our domain $\cD=\cD(\cM)$ we need to generalize  Lemma~\ref{lemma3.2R} with  $\rho(x)$, the distance between $x$ and $\partial\cD$. This will be done in  Lemma~\ref{lemma3.2S}. To prove Lemma~\ref{lemma3.2S}, we use Lemma~\ref{lemma3.2R} and the following  two auxiliary lemmas.

\begin{lemma}\label{lemma.distance}
Let $\partial^S\cM$ denote the boundary of $\cM$ in $S^{d-1}$.

\begin{enumerate}[align=right,label=\textup{(\roman*)}]
\item\label{lemma.distance.1} For any $x\in S^{d-1}$,
\begin{align*}
d(x,\partial \mathcal{D})&\leq d(x,\partial^S \mathcal{M})\leq 2\,d(x,\partial\mathcal{D}).
\end{align*}

\item\label{lemma.distance.2} Let $0<\delta\leq 1$ and $x,\ y\in\bR^d\setminus\{0\}$. If $$\frac{x\cdot y}{|x||y|}\leq (1-\delta),$$ then $\delta\left(|x|^2+|y|^2\right)\leq |x-y|^2$.
\end{enumerate}
\end{lemma}

\begin{proof}
The second claim \ref{lemma.distance.2} follows a direct calculation and we leave it to the reader.

Let us prove \ref{lemma.distance.1}. Take any $x\in S^{d-1}$. The fact $\partial^S \cM\subset \partial\cD$ in $\bR^d$ instantly implies $$d(x,\partial\mathcal{D})\leq d(x,\partial^S\mathcal{M}).
$$

For the other inequality, we consider two cases of  $d(x,\partial\mathcal{D}) \ (\le |x-0|=1)$. 

If $d(x,\partial\mathcal{D})=1$, we have
$$
d(x,\partial^S\mathcal{M})\leq 2=2\,d(x,\partial\mathcal{D})
$$
since $x\in S^{d-1}$ and $\cM\subset S^{d-1}$.

If $d(x,\partial\mathcal{D})<1$, then we note that there exists $y\in\partial\cD$ satisfying $|y|\neq 0$ and $|x-y|=d(x,\partial\cD)$.
Take the unique $\theta\in[0,\pi]$ satisfying $x\cdot y=|x|\,|y|\,\cos\theta=|y|\cos\theta$. Since $ty$ is on $\partial \cD$ for any $t>0$, the function $f(t):=|x-ty|^2$, $t>0$, has the minimum at $t=1$.
Using
$$
f(t)=1-2t\ x\cdot y +t^2|y|^2=1-2t|y|\cos\theta+t^2|y|^2=|y|^2\left(t-\frac{\cos\theta}{|y|}\right)^2+\sin^2\theta,
$$
we get
$$
|y|=\cos\theta, \quad \theta\in[0,\pi/2], \quad\text{and}\quad|x-y|=\sin\theta.
$$
Hence, we have
$$
d(x,\partial^S\mathcal{M})\leq \Big|x-\frac{y}{|y|}\Big|=\sqrt{2-2\cos\theta}=2\sin\frac{\theta}{2}\leq 2\sin\theta = 2|x-y|=2\, d(x,\partial\mathcal{D}).
$$
\end{proof}

 Recall that $\overline{\cM}^S$ is the closure of $\cM$ in $S^{d-1}$, and $s_0=(0,\cdots,-1)\not \in \overline{\cM}^S$. Denote
   $$
   B^S_r(p):=B_r(p)\cap S^{d-1}, \,\, p\in S^{d-1}.
   $$

\begin{remark}\label{domain.property}
   Denote $R_0:=\frac{1}{2}d(s_0,\overline{\cM}^S)$. Since $\Omega:=\phi(\cM)$ is   of class of $C^2$ in $\bR^{d-1}$,  there exist constants $r_0\in\left(0,1\wedge R_0\right)$ and $N_0>0$ such that for any $p\in\partial^S\cM$ and $V_p:=\phi\left(B^S_{r_0}(p)\right)\subset \bR^{d-1}$,  there exists a $\cC^2$ bijective (flattening boundary) map $\psi_p=(\psi^1_p,\ldots,\psi^{d-1}_p)$ from $V_p$ onto a domain $G_p:=\psi_p(V_p)\subset \bR^{d-1}$ satisfying the following: 
\begin{enumerate}[align=right,label=\textup{(\roman*)}]
\item $\psi_p(V_p\cap \Omega)=G_p\cap\bR^{d-1}_+$ and $\psi_p(y_0)=0$,  where $y_0=\phi(p)$ and
$\bR^{d-1}_+=\{y=(y^1,\ldots,y^{d-1})\in\bR^{d-1}\,:\,y^1> 0\}$.

\item $\psi_p(V_p\cap \partial\Omega)=G_p\cap \partial \bR^{d-1}_+$.

\item for any $y\in V_p$, 
\begin{align*}
N_0^{-1}d(y,\partial\Omega)\leq \left|\psi^1_p(y)\right|\leq N_0 d(y,\partial\Omega).
\end{align*}

\item $\|\psi_p\|_{\cC^2(V_p)}+\|\psi^{-1}_p\|_{\cC^2(G_p)}\leq N_0$ and 
$$
N_0^{-1}|y_1-y_2|\leq |\psi_p(y_1)-\psi_p(y_2)|\leq N_0|y_1-y_2|,\quad \forall \, y_1,y_2\in V_p.
$$
\end{enumerate}
\end{remark}

For the next lemmas,  for open sets $U$ of $S^{d-1}$ and open sets $V$ in $\bR^{d-1}$, we  consider  two  types of  domains $\cD(U)$ and $\widetilde{\cD}(V)$ in $\bR^d$ defined by
$$
\cD(U):=\left\{x\in\bR^d\setminus\{0\}\,:\,\frac{x}{|x|}\in U\right\},
$$
and
\begin{eqnarray*}
\widetilde{\cD}(V)
&=&\left\{x=r(\xi',1)=r(\xi^1,\cdots,\xi^{d-1},1)\,: \,r>0, \,\xi'\in V \right\}.
\end{eqnarray*}
Obviously, $x=(x',x^d)\in \tilde{\cD}(V)$ if and only if $x^d>0$ and $\frac{x'}{x^d}\in V$.

Now we recall the stereographic projection $\phi$ described in  Section 2 and also take   the constant $r_0$ from  Remark~\ref{domain.property}. Then for any 
fixed $p\in\partial^S\cM$, 
 let $V_p:=\phi\left(B^S_{r_0}(p)\right)$ and $G_p:=\psi_p(V_p)$ with the map $\psi_p$ described in Remark ~\ref{domain.property}. 
Then we can define the following  two bijective maps:
\begin{align*}
\Phi_p:\cD\left(B_r^S(p)\right)&\rightarrow\qquad \widetilde{\cD}\left(V_p\right)\\
x\qquad&\mapsto \left(|x|\,\phi\Big(\frac{x}{|x|}\Big),|x|\right),
\end{align*}
and
\begin{align*}
\Psi_p:\tilde{\cD}\left(V_p\right)&\rightarrow\qquad \widetilde{\cD}\left(G_p\right)\\
(y',y^d)&\mapsto \left(y^d\,\psi_p \Big(\frac{y'}{y^d}\Big),y^d\right).
\end{align*}
Note
$$
\Psi_p \circ \Phi_p (x)=\left(|x|\,\psi_p\circ \phi\Big(\frac{x}{|x|}\Big), |x|\right).
$$

\begin{lemma}\label{function.property}
There exists a constant $N=N(\cM,d)>0$ such that for any $p\in \partial^S\cM$ and the maps $\Phi_p$ and $\Psi_p$, 

\begin{enumerate}[align=right,label=\textup{(\roman*)}]
\item\label{function.property.1} $\,N^{-1}|x-y|\leq |(\Psi_p\circ\Phi_p)(x)-(\Psi_p\circ\Phi_p)(y)|\leq N|x-y|$,

\item\label{function.property.2} $\,N^{-1}\leq |det\,D\left(\Psi_p\circ\Phi_p\right)(x)|\leq N,$

\item\label{function.property.3} $\,N^{-1}\,d(x,\partial\mathcal{D})\leq (\Psi_p\circ\Phi_p)^1(x)\leq N\, d(x,\partial\mathcal{D})$,

\item\label{function.property.4} $\,N^{-1}\,|x|\leq |(\Psi_p\circ\Phi_p)(x)|\leq N\, |x|$
\end{enumerate}
for all $x,y\in \cD\left(B^{S}_{r_0}(p)\right)$, where $D\left(\Psi_p\circ\Phi_p\right)$ is the Jacobian matrix function of  $\Psi_p\circ\Phi_p$.
\end{lemma}

\begin{proof} 
  We note that there exists a constant $R_1>0$, which depends only on the constants $R_0$ and $N_0$ in Remark~\ref{domain.property}, such that 
$$
V_p,\, G_p\subset\left\{x\in\bR^{d-1}\,:\,|x|<R_1\right\}
$$
for any  $p\in \partial^S\cM$. For  $p\in \partial^S\cM$, let us define $U_p:=B^S_{r_0}(p)$.

\ref{function.property.1} For any fixed $p\in \partial^S\cM$, by the definition of $\Psi_p$ and Remark \ref{domain.property}, for  any $\xi,\eta\in \widetilde{\cD}(V_p)$ we have
\begin{align*}
|\Psi_p(\xi)-\Psi_p(\eta)|&\leq |\xi^d|\cdot\left|\psi_p\Big(\frac{\xi'}{\xi^d}\Big)-\psi_p\Big(\frac{\eta'}{\eta^d}\Big)\right|+|\xi^d-\eta^d|\cdot\left|\psi_p\Big(\frac{\eta'}{\eta^d}\Big)\right|+|\xi^d-\eta^d|\\
&\leq N_0\,\left|\xi'-\frac{\xi^d}{y^d}\,\eta'\right|+(R_1+1)|\xi-\eta|\\
&\leq N_0\,\left(|\xi'-\eta'|+\left|\frac{\eta'}{\eta^d}\right|\,|\eta^d-\xi^d|\right)+(R_1+1)|\xi-\eta|\\
&\leq N(N_0,R_1)\,|\xi-\eta|,
\end{align*}
where $\xi=(\xi',\xi^d)$, $\eta=(\eta',\eta^d)$, and $N_0, R_1$ together with $R_0$ are the constants from Remark \ref{domain.property}.
Adding the same calculation for $\Psi_p^{-1}$, we find that there exists a constant $N=N(\cM)>0$ so that 
\begin{align*}
N^{-1}|\xi-\eta|\leq |\Psi_p(\xi)-\Psi_p(\eta)|\leq N|\xi-\eta|
\end{align*}
for any $p\in \partial^S\cM$ and  $\xi,\eta\in \widetilde{\cD}(V_p)$.
On the othe hand, by the definition of $\Phi_p$, for fixed $x,y\in\cD(U_p)$  we get
\begin{align*}
|\Phi_p(x)-\Phi_p(y)|&\leq |x|\cdot\left|\phi\Big(\frac{x}{|x|}\Big)-\phi\Big(\frac{y}{|y|}\Big)\right|+\big|\,|x|-|y|\,\big|\cdot\left|\phi\Big(\frac{y}{|y|}\Big)\right|+\big|\,|x|-|y|\,\big|\\
&\leq N(R_0)\left|x-\frac{|x|}{|y|}y\right|+(R_1+1)|x-y|\\
&\leq N\,|x-y|,
\end{align*}
where $N$  depends only on $R_0$ and $R_1$.

 For the reverse inequality,  we first note that $\phi^{-1}\in \cC^{\infty}(\bR^{d-1})$. Hence, there exists a constant $N=N(R_1,d)$ such that 
\begin{align*}
|\phi^{-1}(\xi')-\phi^{-1}(\eta')|\leq N\,|\xi'-\eta'|
\end{align*} 
for all $\xi',\eta'\in \bR^{d-1}$ with $|\xi'|,\,|\eta'|\leq R_1$.
By a similar calculation as above, we get
\begin{align*}
|\Phi_p^{-1}(\xi)-\Phi_p^{-1}(\eta)|\leq N\,|\xi-\eta|
\end{align*}
for any  $\xi,\,\eta\in\tilde{\cD}(V_p)$. This implies 
$$
|x-y|\le N |\Phi_p(x)-\Phi_p(y)|
$$
for any $x,y\in\cD(U_p)$. 

Gathering all, we conclude that   there exists a constant $N=N(\cM,d)>0$ such that $$N^{-1}|x-y|\leq |(\Psi_p\circ\Phi_p)(x)-(\Psi_p\circ\Phi_p)(y)|\leq N|x-y|$$
 for any $p\in\partial^S\cM$ and  $x,y\in\cD(U_p)$.

\ref{function.property.2} By the result of \ref{function.property.1}, there exists a constant $N$ such that 
\begin{align*}
\sup_{\cD\left(U_p\right)}|D(\Psi_p\circ\Phi_p)|+\sup_{\widetilde{\cD}(G_p)}|D(\Phi_p^{-1}\circ\Psi_p^{-1})|\leq N.
\end{align*}
This gives
\begin{align*}
N^{-1}\leq |det\,D(\Psi_p\circ\Phi_p)(x)|\leq N, \quad \forall\, x\in \cD(B^{S}_{r_0}(p))
\end{align*}
for some constant $N=N(\cM,d)>0$.

To prove \ref{function.property.3} and \ref{function.property.4}, we first recall that
\begin{align*}
(\Psi_p\circ\Phi_p)(x)=\left(|x|\,(\psi_p\circ\phi)\Big(\frac{x}{|x|}\Big),|x|\right).
\end{align*}

\ref{function.property.3} By Remark~\ref{domain.property}, we have
\begin{align*}
(\Psi_p\circ\Phi_p)^1(x)=|x|(\psi_p^1\circ\phi)(\hat{x})\sim |x|d\left(\phi(\hat{x}),\partial\Omega\right)
\end{align*}
for  $x\in\cD(U_p)$, where $\hat{x}:=\frac{x}{|x|}$ and $\Omega=\phi(\cM)$.  Since, $\phi(\partial^S\cM)=\partial\Omega$, the result of (i) and Lemma~\ref{lemma.distance} give
$$
d(\phi(\hat{x}),\partial\Omega)\sim d(\hat{x},\partial^S\cM)\sim d(\hat{x},\partial\cD).
$$
Consequently, we get
$$
(\Psi_p\circ\Phi_p)^1(x) \sim |x|d(\phi(\hat{x}),\partial\Omega) \sim |x|d(\hat{x},\partial\cD)=d(x,\partial\cD).
$$
All the comparabilities above depend only on $\cM$ and $d$.

\ref{function.property.4} This is due to
\begin{align*}
|(\Psi_p\circ\Phi_p)(x)|\sim |x|\left(|(\psi_p\circ\phi)(\hat{x})|+1\right),\quad
|(\psi_p\circ\phi)(\hat{x})|\leq C <\infty
\end{align*}
for  $x\in \cD(U_p)$, where the comparability relation and the constant $C$ depend only on $\cM$, $d$.
\end{proof}

\begin{lemma}\label{lemma3.2S}
Let $\sigma>0,\ \alpha+\gamma>-d$, $\gamma>-1$ and $\beta,\ \omega \in \mathbb{R}$. Then there exists a constant $N(\cM,d, \alpha, \beta, \gamma,\omega,\sigma)>0$ such that 
\begin{align*}
\int_{\mathcal{D}} \frac{|y|^{\alpha}}{\left(|y|+1\right)^{\beta}}\frac{\rho(y)^{\gamma}}{\left(\rho(y)+1\right)^{\omega}}\ e^{-\sigma |x-y|^2} dy \leq N \left(|x|+1\right)^{\alpha-\beta}\left(\rho(x)+1\right)^{\gamma-\omega}
\end{align*}
 for any $x\in\cD$.
\end{lemma}

\begin{proof}
{\bf 1}. Take the constant $r_0$ in Remark ~\ref{domain.property} and let
$$
F:=\left\{p\in S^{d-1}\,:\, d(p,\partial^S\mathcal{M})\leq \frac{r_0}{4}\right\}.
$$
Then there exist  a finite number of points $p_1,\,\cdots,\,p_m\in \partial^S \cM$ such that 
$$
F\subset \bigcup_{i=1}^mB_{r_0/2}^S(p_i)\,.
$$
Denote
$$
\mathcal{D}_0=\cD\left(\cM\setminus F\right)=\left\{z\in\mathcal{D}\,:\,d\Big(\frac{z}{|z|},\partial^S\cM\Big)>\frac{r_0}{4} \right\}
$$
and
$$
\mathcal{D}_i:=\mathcal{D}\left(B_{r_0/2}^S(p_i)\right)=\left\{z\in\mathcal{D}\,:\,d\Big(\frac{z}{|z|},p_i\Big)<\frac{r_0}{2}\right\}
$$
for $i=1,\,\cdots,\,m$. It is obvious that
$$
\cD=\bigcup_{i=0}^m\cD_i.
$$

Now, we  fix $x\in \cD$ and  consider two parts of $\cD$:
\begin{align*}
E_1(x) = \left\{y\in\mathcal{D}\ :\ \frac{x}{|x|}\cdot \frac{y}{|y|}>1-\delta \right\} \quad \text{and}\quad E_2(x) = \left\{y\in\mathcal{D}\ :\ \frac{x}{|x|}\cdot \frac{y}{|y|}\leq 1-\delta \right\}
\end{align*}
with $\delta=\frac{r_0^2}{128}$. 
Note that if $y\in E_1(x)$, then 
$$
\left|\frac{x}{|x|}-\frac{y}{|y|}\right|<\frac{r_0}{8}.
$$
Now, we consider
\begin{align*}
I(x):=&\int_{\cD} \frac{|y|^{\alpha}}{\left(|y|+1\right)^{\beta}}\frac{\rho(y)^{\gamma}}{\left(\rho(y)+1\right)^{\omega}}\ e^{-\sigma |x-y|^2}dy\\
=&\,I_1(x)+I_2(x)
\end{align*}
where
\begin{align*}
I_i(x)=\int_{E_i(x)} \frac{|y|^{\alpha}}{\left(|y|+1\right)^{\beta}}\frac{\rho(y)^{\gamma}}{\left(\rho(y)+1\right)^{\omega}}\ e^{-\sigma |x-y|^2}dy, \quad i=1,2.
\end{align*}

{\bf 2}. Estimation of  $I_1(x)$.

- {\bf{Case 1.}}
$x\in \mathcal{D}_0$.

For $y\in E_1(x)$, 
$$
d\Big(\frac{y}{|y|},\partial^S\cM\Big)\geq d\Big(\frac{x}{|x|},\partial^S\cM\Big)-\left|\frac{x}{|x|}-\frac{y}{|y|}\right|\geq\frac{r_0}{8}.
$$
Therefore, we get 
$$
d\Big(\frac{x}{|x|},\partial^S\cM\Big),\, d\Big(\frac{y}{|y|},\partial^S\cM\Big)\in \left[\frac{r_0}{8},2\right].
$$
By the observation $\rho(x)=|x|\rho(\frac{x}{|x|})$ and Lemma~\ref{lemma.distance}, there exists a constant $N=N(\cM,d)>0$ such that 
\begin{align*}
N^{-1}|x|\leq\rho(x)\leq N|x|\quad\text{and}\quad N^{-1}|y|\leq\rho(y)\leq N|y|\,.
\end{align*} 
By Lemma~\ref{lemma3.2R}, we get
\begin{align*}
I_1(x)&\leq N\int_{\mathcal{D}_0} \frac{|y|^{\alpha+\gamma}}{\left(|y|+1\right)^{\beta+\omega}}\ e^{-\sigma |x-y|^2}dy\\
&\leq N\int_{\mathbb{R}^d} \frac{|y|^{\alpha+\gamma}}{\left(|y|+1\right)^{\beta+\omega}}\ e^{-\sigma |x-y|^2}dy\\
&\leq N\left(|x|+1\right)^{\alpha-\beta+\gamma-\omega}\\
&\leq N \left(|x|+1\right)^{\alpha-\beta}\left(\rho(x)+1\right)^{\gamma-\omega},
\end{align*}
where $N$ depends only on $d$, $\alpha$, $\beta$, $\gamma$, $\omega$, $\sigma$ and $\cM$.

- {\bf{Case 2.}}
$x\in\bigcup_{i=1}^m\mathcal{D}_i$.

Without loss of generality, we assume  $x\in\cD_1$. Then, 
for $y\in E_1(x)$, we have
$$
\left|\frac{x}{|x|}-p_1\right|<\frac{r_0}{2},\quad \text{and hence}\quad\left|\frac{y}{|y|}-p_1\right|<r_0.
$$
Also, note that
\begin{align*}
&|x-y|\sim |(\Psi_{p_1}\circ\Phi_p)(x)-(\Psi_{p_1}\circ\Phi_p)(y)|,\\
&\rho(x)\sim (\Psi_{p_1}\circ\Phi_{p_1})^1(x),\quad \rho(y)\sim(\Psi_{p_1}\circ\Phi_{p_1})^1(y),\\
&|x|\sim |(\Psi_{p_1}\circ\Phi_{p_1})(x)|,\quad |y|\sim |(\Psi_{p_1}\circ\Phi_{p_1})(y)|
\end{align*}
for $y\in E_1(x)$, where all the comparability relations depend only on $\cM,d$.
By the change of variables  $z=(\Psi_{p_1}\circ\Phi_{p_1})(y)$, we have
\begin{align*}
I_1(x)&=\int_{E_1(x)} \frac{|y|^{\alpha}}{\left(|y|+1\right)^{\beta}}\frac{\rho(y)^{\gamma}}{\left(\rho(y)+1\right)^{\omega}}\ e^{-\sigma |x-y|^2} dy\\
&\leq\int_{\cD\left(B^S_{r_0}(p_1)\right)} \frac{|y|^{\alpha}}{\left(|y|+1\right)^{\beta}}\frac{\rho(y)^{\gamma}}{\left(\rho(y)+1\right)^{\omega}}\ e^{-\sigma |x-y|^2} dy\\
&\leq N \int_{\mathbb{R}^d}\frac{|z|^{\alpha}}{\left(|z|+1\right)^{\beta}}\frac{|z^1|^{\gamma}}{\left(|z^1|+1\right)^{\omega}}\ e^{-\sigma' |z^*-z|^2}dz,
\end{align*}
where $z^*:=(\Psi_{p_1}\circ\Phi_{p_1})(x)$ and $N$ and $\sigma'$ depend only on $\cM$, $d$, and $\sigma$.
Lastly, since  $\alpha+\gamma>-d$, $\gamma>-1$, Lemma~\ref{eq.lemma3.2R}  yields
\begin{align*}
\int_{\mathbb{R}^d_+}\frac{|z|^{\alpha}}{\left(|z|+1\right)^{\beta}}\frac{(z^1)^{\gamma}}{\left(z^1+1\right)^{\omega}}\ e^{-\sigma' |z^*-z|^2}dz
&\leq N (|z^*|+1)^{\alpha-\beta}((z^*)^1+1)^{\gamma-\omega}\\
&\leq N (|x|+1)^{\alpha-\beta}(\rho(x)+1)^{\gamma-\omega}
\end{align*}
with $N=N(\cM,d,\alpha,\beta,\gamma,\omega,\sigma)$.
Hence, we get 
\begin{align*}
I_1(x)\leq N\left(|x|+1\right)^{\alpha-\beta}\left(\rho(x)+1\right)^{\gamma-\omega}.
\end{align*}

{\bf 3}. Estimation of  $I_2(x)$.

Since $\frac{x}{|x|}\cdot\frac{y}{|y|}\leq 1-\delta$ for $y\in E_2(x)$, Lemma~\ref{lemma.distance} gives 
$$
e^{-\sigma|x-y|^2}\leq e^{-\sigma'|x|^2}\cdot e^{-\sigma'|y|^2}
$$
where $\sigma'=\sigma\delta$.
Therefore we have
\begin{align*}
I_2(x)&=\int_{E_2(x)} \frac{|y|^{\alpha}}{\left(|y|+1\right)^{\beta}}\frac{\rho(y)^{\gamma}}{\left(\rho(y)+1\right)^{\omega}}\ e^{-\sigma |x-y|^2} dy\\
&\leq  e^{-\sigma' |x|^2} \int_{\mathcal{D}} \frac{|y|^{\alpha}}{\left(|y|+1\right)^{\beta}}\frac{\rho(y)^{\gamma}}{\left(\rho(y)+1\right)^{\omega}}\ e^{-\sigma' |y|^2} dy \\
&\leq e^{-\sigma' |x|^2}\sum_{i=0}^m \int_{\mathcal{D}_i} \frac{|y|^{\alpha}}{\left(|y|+1\right)^{\beta}}\frac{\rho(y)^{\gamma}}{\left(\rho(y)+1\right)^{\omega}}\ e^{-\sigma' |y|^2} dy.
\end{align*}
Following calculations used for $I_1(x)$, we get
\begin{align*}
I_2(x)\leq N(\alpha,\beta,\gamma,\omega,\sigma')\, e^{-\sigma'|x|^2}.
\end{align*}

Now,  note that for any fixed  $\sigma_1>0,\ \sigma_2\in\mathbb{R}$,
\begin{align*}
\left(a+1\right)^{\sigma_2}e^{-\sigma_1 a^2},\quad  a>0
\end{align*}
is bounded by a constant depending only on $\sigma_1,\sigma_2$, and also note
\begin{align*}
1\leq \rho(x)+1\leq |x|+1.
\end{align*}
Putting
\begin{equation*}
\sigma_2=
\begin{cases}
-\left(\alpha-\beta\right)               & \text{if }\gamma-\omega>0\\
-\left(\alpha-\beta+\gamma-\omega\right) & \text{otherwise},
\end{cases}
\end{equation*}
we conclude that there exists a constant $N=N(\alpha,\beta,\gamma,\omega,\sigma')>0$ such that 
\begin{align*}
e^{-\sigma'|x|^2}\leq N \left(|x|+1\right)^{\alpha-\beta}\left(\rho(x)+1\right)^{\gamma-\omega}.
\end{align*}
Hence, we obtain
\begin{align*}
I_2(x)\leq N \left(|x|+1\right)^{\alpha-\beta}\left(\rho(x)+1\right)^{\gamma-\omega},
\end{align*}
where $N$ depends only on $\cM, \alpha$, $\beta$, $\gamma$, $\omega$, and $\sigma$.
\end{proof}

Next we introduce what we prepared in \cite{Green} for our main result of this article. It is a refined estimate of the Green's function of the parabolic operator $\partial_t-\cL$ with the domain $\cD=\cD(\cM)$.

\begin{thm} \label{green}
   Let   $\lambda^+\in(0,\lambda^+_{c,\cL})$,  $\lambda^-\in(0,\lambda^-_{c,\cL})$, and denote $K^+_0:=K_0(\cL,\cM,\lambda^+)$ and  $K^-_0:=K_0(\hat{\cL},\cM,\lambda^-)$, where $K_0$ is from Definition \ref{lambda}.  Then there exist positive constants 
   $N=N(\mathcal{M}, \nu_1,\nu_2, \lambda^{\pm},K_0^{\pm})$ and $\sigma=\sigma(\nu_1,\nu_2)$  such that
\begin{equation}\label{key estimate}
  G(t,s,x,y)\leq \frac{N}{(t-s)^{d/2}}\,  R^{\lambda^+ -1}_{t-s,x}\, R^{\lambda^- -1}_{t-s,y}\,  J_{t-s,x} \, J_{t-s,y}\, e^{-\sigma \frac{|x-y|^2}{t-s}}
\end{equation}
   for any $t>s$, $x,y\in\mathcal{D}$.   Moreover, if
\begin{equation}
   \label{eqn 8.12.77}
\lambda^+,\;\lambda^- \in \big(0,\lambda_{c}(\nu_1,\nu_2)\big),
\end{equation}
then the constant $N$ in (\textsl{\ref{key estimate}}) can  depend only on $\cM,\nu_1, \nu_2$, and $\lambda^{\pm}$.\end{thm}

\begin{proof}
(3.4) holds due to  \cite[Theorem 2.6]{Green}.  If \eqref{eqn 8.12.77} holds, then  by  Definition  \ref{def 8.12}  the constant $K_0^{\pm}$  can be chosen such that it depends only on $\cM,\nu_1,\nu_2$ and 
$\lambda^{\pm}$.
\end{proof}

We are ready to prove Lemma~\ref{main est}.
\vspace{0.2cm}

\begin{proof}[Proof of Lemma~\ref{main est}]\quad
\vspace{0.2cm}

We only prove the lemma for the case    $\lambda^+\in(0,\lambda^+_{c,\cL})$ and  $\lambda^-\in(0,\lambda^-_{c,\cL})$. This is because the same proof works under condition   \eqref{eqn 8.12.77} without any changes. The difference of the dependency of constant $N$ in \eqref{main inequality} is inherited from the constant $N$ in \eqref{key estimate}.

\vspace{1mm}

{\bf 1}. Denote $\mu:=1+(\theta-d)/p$ and $\alpha:=1+(\Theta-d)/p$.
By the range of $\theta$ given in the statement, we can always find $\lambda^+<\lambda^+_c$ and $\lambda^-<\lambda^-_c$ satisfying
\begin{align*}
2-\frac{d}{p}-\lambda^+<\mu<d-\frac{d}{p}+\lambda^-.
\end{align*}
Also, by the range of $\Theta$  we have
\begin{align*}
1-\frac{1}{p}<\alpha<2-\frac{1}{p}.
\end{align*}
By the designed range of $\mu$ and $\alpha$, we can choose and fix the constants $\gamma_1$, $\gamma_2$, $\omega_1$, and $\omega_2$ satisfying
\begin{align*}
-\frac{d-2}{p} < \gamma_1 < \lambda^+ - 2 + \mu+\frac{2}{p}\,,&\qquad 0<\gamma_2<\lambda^-+d-\frac{d}{p}-\mu\\
\frac{1}{p}<\omega_1<\alpha-1+\frac{2}{p}\,,\qquad\,\,&\qquad 0<\omega_2<2-\frac{1}{p}-\alpha.
\end{align*}
Since $\lambda^+<\lambda^+_c$ and $\lambda^-<\lambda^-_c$, by Theorem~\ref{green}, there exist constants $N,\sigma>0$ such that
\begin{align*}
G(t,s,x,y)\leq& \frac{N}{(t-s)^{d/2}}e^{-\sigma\frac{|x-y|^2}{t-s}}\,J_{t-s,x} J_{t-s,y}R^{\lambda^+-1}_{t-s,x}R^{\lambda^--1}_{t-s,y}\\
=& \frac{N}{(t-s)^{d/2}}e^{-\sigma\frac{|x-y|^2}{t-s}}\,R^{\gamma_1}_{t-s,x}\left(\frac{J_{t-s,x}}{R_{t-s,x}}\right)^{\omega_1}R^{\gamma_2}_{t-s,y}\left(\frac{J_{t-s,y}}{R_{t-s,y}}\right)^{\omega_2}\\
&\qquad\qquad\qquad\qquad\times R^{\lambda^+-\gamma_1}_{t-s,x}\left(\frac{J_{t-s,x}}{R_{t-s,x}}\right)^{1-\omega_1}R^{\lambda^--\gamma_2}_{t-s,y}\left(\frac{J_{t-s,y}}{R_{t-s,y}}\right)^{1-\omega_2}
\end{align*}
for all $t>s$ and $x,y\in\cD$.

{\bf 2}. We set
$$
h(t,x)=|x|^{\mu}\left(\frac{\rho(x)}{|x|}\right)^\alpha f(t,x)\quad ;\quad h=\rho_{\circ}^{\mu-\alpha}\rho^{\alpha}f.
$$
Then, because of $\mu=1+(\theta-d)/p$ and $\alpha=1+(\Theta-d)/p$, the terms in  estimate \eqref{main inequality} turn into
\begin{align*}
\int^T_0\int_{\mathcal{D}}\left|\rho(x)\ f(t,x)\right|^p \rho_{\circ}^{\theta-d}(x)\left(\frac{\rho(x)}{\rho_{\circ}(x)}\right)^{\Theta-d}dxdt&=\|h\|^p_{L_p\left([0,T]\times \mathcal{D}\right)}.\\
\int^T_0\int_{\mathcal{D}}\left|\rho^{-1}(x)u(t,x)\right|^p \rho_{\circ}^{\theta-d}(x)\left(\frac{\rho(x)}{\rho_{\circ}(x)}\right)^{\Theta-d}dx\,dt&=\Big\|\rho_{\circ}^{\mu-\alpha}\rho^{\alpha-2}u\Big\|_{L_p\left([0,T]\times \mathcal{D}\right)}
\end{align*}
for the function
\begin{equation}\label{the solution}
u(t,x):=\int_0^t\int_{\cD}G(t,s,x,y)f(s,y)dyds,
\end{equation}
 and hence we need to show
\begin{equation}\label{main inequality 2}
\Big\|\rho_{\circ}^{\mu-\alpha}\rho^{\alpha-2}u\Big\|_{L_p\left([0,T]\times \mathcal{D}\right)}\le N \|h\|_{L_p\left([0,T]\times \mathcal{D}\right)}.
\end{equation}

We start with, using H\"older inequality, 
\begin{align*}
|u(t,x)|&=\left|\int^t_0\int_{\mathcal{D}}G(t,s,x,y)f(s,y)dyds\right|\\
&\leq \int^t_0\int_{\mathcal{D}}G(t,s,x,y)|y|^{-\mu+\alpha}\rho(y)^{-\alpha}|h(s,y)|dyds\\
&\leq N\cdot I_1(t,x)\cdot I_2(t,x),
\end{align*}
where $q=p/(p-1)$,
$$
I_1^p(t,x)=\int^t_0\int_{\mathcal{D}}\frac{1}{(t-s)^{d/2}}\,e^{-\sigma\frac{|x-y|^2}{t-s}}K_{1,1}(t-s,x)K_{1,2}(t-s,y)|h(s,y)|^pdyds,
$$
and
$$
I_2^q(t,x)=\int^t_0\int_{\mathcal{D}}\frac{1}{(t-s)^{d/2}}\,e^{-\sigma\frac{|x-y|^2}{t-s}}K_{2,1}(t-s,x)K_{2,2}(t-s,y)|y|^{(-\mu+\alpha)q}\rho^{-\alpha q}(y)dyds
$$
with
$$
 K_{1,1}(t,x)=R^{\gamma_1p}_{t,x}\left(\frac{J_{t,x}}{R_{t,x}}\right)^{\omega_1p},\quad K_{1,2}(t,y)
 =R^{\gamma_2 p}_{t,y}\left(\frac{J_{t,y}}{R_{t,y}}\right)^{\omega_2 p},
$$
$$
K_{2,1}(t,x)=R^{(\lambda^+-\gamma_1)q}_{t,x}\left(\frac{J_{t,x}}{R_{t,x}}\right)^{(1-\omega_1)q},\quad \, K_{2,2}(t,y)=R^{(\lambda^--\gamma_2)q}_{t,y}\left(\frac{J_{t,y}}{R_{t,y}}\right)^{(1-\omega_2)q}.
$$
 
{\bf 3}. In this step, we will show that ther exists a constant $N>0$ such that 
$$
I_2(t,x)\leq N\,|x|^{-\mu+\alpha}\rho(x)^{-\alpha+\frac{2}{q}}.
$$
In particular, the right hand side is independent of $t$.

Since $(\lambda^--\mu-\gamma_2)q>-d$ and $(1-\alpha-\omega_2)q>-1$, by change of variables and Lemma~\ref{lemma3.2S}, we get
\begin{align*}
&\quad\frac{1}{(t-s)^{d/2}}\int_{\mathcal{D}}e^{-\sigma\frac{|x-y|^2}{t-s}}K_{2,2}(t-s,y)|y|^{(-\mu+\alpha)q}|\rho(y)|^{\alpha q}dy\\
&=(t-s)^{-\mu q/2}\int_{\mathcal{D}}e^{-\sigma|\frac{x}{\sqrt{t-s}}-y|^2}\frac{|y|^{(\lambda^--\mu-\gamma_2-1+\alpha+\omega_2)q}}{(|y|+1)^{(\lambda^--\gamma_2-1+\omega_2)q}}\cdot\frac{\rho(y)^{(1-\alpha-\omega_2)q}}{(\rho(y)+1)^{(1-\omega_2)q}}dy\\
&\leq  N \left(|s|+\sqrt{t-s}\right)^{(-\mu+\alpha)q}\left(\rho(x)+\sqrt{t-s}\right)^{-\alpha q}.
\end{align*}
Hence, we have
\begin{align*}
I_{2}^q(t,x)\,&\leq N\int^t_0K_{2,1}(t-s,x)\cdot\left(|x|+\sqrt{t-s}\right)^{(-\mu+\alpha)q}\left(\rho(x)+\sqrt{t-s}\right)^{-\alpha q}ds\\
&\leq N\int^{\infty}_{0}\frac{|x|^{(\lambda^+-1-\gamma_1+\omega_1)q}}{(|x|+\sqrt{s})^{(\lambda^+-1+\mu-\gamma_1-\alpha+\omega_1)q}}\cdot\frac{\rho(x)^{(1-\omega_1)q}}{(\rho(x)+\sqrt{s})^{(\alpha+1-\omega_1)q}}ds.
\end{align*}
Moreover, since $(\lambda^++\mu-\gamma_1)q > 2$ and $(\alpha+1-\omega_1)q> 2$, by Lemma~\ref{lemma3.1} we further obtain
\begin{align*}
I_2^q(t,x)&\leq N |x|^{(-\mu+\alpha)q}\rho(x)^{-\alpha q+2}.
\end{align*}
This implies
\begin{align*}
|u(t,x)|\leq N\, I_1(t,x)\cdot I_2(t,x)\leq N\, |x|^{-\mu+\alpha}\rho(x)^{-\alpha+\frac{2}{q}}\, I_1(t,x),
\end{align*}
and hence
\begin{align*}
|x|^{\mu-\alpha}\rho(x)^{\alpha-2}|u(t,x)|\leq N \rho(x)^{-\frac{2}{p}}I_1(t,x).
\end{align*}

{\bf 4}. Using this, we have
\begin{align*}
\|\rho_{\circ}^{\mu-\alpha}\rho^{\alpha-2}u\|^p_{L_p([0,T]\times\mathcal{D})}&\leq N \int_0^T\int_{\mathcal{D}}|\rho(x)|^{-2}I_1^p(t,x)\,dxdt\\
&=N\int_0^T\int_{\mathcal{D}}I_3(s,y)\cdot |h(s,y)|^pdyds,
\end{align*}
where
\begin{align*}
I_3(s,y)=\int^T_s\int_{\mathcal{D}}\frac{1}{(t-s)^{d/2}}e^{-\sigma\frac{|x-y|^2}{t-s}}K_{1,1}(t-s,x)K_{1,2}(t-s,y)\rho(x)^{-2}\,dxdt.
\end{align*}
Since $\gamma_1p-2>-d$ and $\omega_1p-2>-1$, by change of variables and Lemma~\ref{lemma3.2S},
\begin{align*}
I_3(s,y)&=\int^T_s\frac{1}{(t-s)^{d/2}}K_{1,2}(t-s,y)\left(\int_{\mathcal{D}}e^{-\sigma\frac{|x-y|^2}{t-s}}K_{1,1}(t-s,x)\rho(x)^{-2}\,dx\right)dt\\
&\leq \int^{\infty}_0\frac{1}{t}K_{1,2}(t,y)\left(\int_{\mathcal{D}}\frac{|x|^{(\gamma_1-\omega_1)p}}{(|x|+1)^{(\gamma_1-\omega_1)p}}\frac{\rho(x)^{\omega_1p-2}}{(\rho(x)+1)^{\omega_1p}}e^{-\sigma'|x-\frac{y}{\sqrt{t}}|^2}\,dx\right)dt\\
&\leq N\int_0^{\infty}K_{1,2}(t,y)\left(\rho(y)+\sqrt{t}\right)^{-2}dt\\
&= N\int_0^{\infty}\frac{|y|^{(\gamma_2-\omega_2)p}}{\left(|y|+\sqrt{t}\right)^{(\gamma_2-\omega_2)p}}\cdot\frac{\rho(y)^{\omega_2p}}{\left(\rho(y)+\sqrt{t}\right)^{\omega_2p+2}}dt.
\end{align*}
Lastly, owing to $\gamma_2p>$ and $\omega_2p> 0$, Lemma~\ref{lemma3.1} gives
\begin{align*}
I_3(s,y)\leq N(\cM,d,p,\theta,\Theta,\cL).
\end{align*}
Hence, there exists a constant $N=N(\cM,d,p,\theta,\Theta,\cL)>0$ such that 
\begin{align*}
\left\|\rho_{\circ}^{\mu-\alpha}\rho^{\alpha-2}u\right\|^p_{L_p([0,T]\times\mathcal{D})}\leq N \|h\|^p_{L_p([0,T]\times\mathcal{D})}(=N\|\rho_{\circ}^{\mu-\alpha}\rho^{\alpha}f\|^p_{L_p([0,T]\times\mathcal{D})})
\end{align*}
for any $f$ and the corresponding function $u$ in the form \eqref{the solution}. This inequality is \eqref{main inequality 2}. The lemma is proved.
\end{proof}

\mysection{Estimate of high order derivatives}

In this section we estimate weighted $L_p$-norm of derivatives of  solutions to  equation \eqref{heat eqn}. This result is presented in Theorem \ref{regularity.induction} and the proof is based on  an alternative definition of $K^n_{p,\theta,\Theta}(\cD)$ introduced below. 

We start with  weighted Sobolev space $H^n_{p,\Theta}(\cD)$  introduced  in   \cite{KK2004, Krylov 1999-3, Krylov 1999-1, Lo1}. For any $p\in(1,\infty)$ and $\Theta\in \bR$,  denote
\begin{align*}
L_{p,\Theta}(\cD):=L_p(\cD,\rho^{\Theta-d} dx;\bR),
\end{align*}
and  by $H^n_{p,\Theta}(\cD)$, $n=0,1,2,\cdots$, we denote the  space of all functions $f$ satisfying 
\begin{equation*}
\label{eqn 4.24.1}
\|f\|^p_{H^n_{p,\Theta}(\cD)}:=\sum_{|\alpha|\leq n}\|\rho^{|\alpha|}D^{\alpha}f\|^p_{L_{p,\Theta}(\cD)}<\infty.
\end{equation*} 
As described below, the space $H^n_{p,\Theta}(\cD)$ enjoys another definition which suits our purpose well and  also leads  us to an alternative definition of $K^n_{p,\theta,\Theta}(\cD)$.

Let us fix an infinitely differentiable function $\psi$  (e.g. \cite[Lemma 4.13]{Ku}) defined on $\cD$ such that
\begin{equation}
                                 \label{eqn 2.26.1}
N^{-1}\rho(x)\leq \psi(x)\leq N \rho(x),\quad \rho^{m}|D^{m}\psi_x|\leq
N(m)<\infty,\,\, m=0,1,2,\ldots.
\end{equation}
For instance, by mollifying the indicator function of $\{x\in \cD: e^{-1-k}<\rho(x)<e^{1-k}\}:=G_k$ one can easily construct $\xi_k$ such that
$$
\xi_k \in C^{\infty}_0(G_k), \quad |D^m\xi_k|\leq N(m)e^{mk}, \quad \sum_{k\in \bZ} \xi_k(x) \sim 1,
$$
and then one can take
$$
\psi(x):=\sum_{k\in \bZ}e^{-k}\xi_k(x).
$$

We also fix  a nonnegative function  $\zeta\in C^{\infty}_{0}(\bR_{+})$ 
satisfying
\begin{equation}
                                                       \label{11.4.1}
\sum_{k=-\infty}^{\infty}\zeta(e^{k+t})>c>0,\quad\forall\; t\in\bR.
\end{equation}
Note that any non-negative  function $\zeta\in C^{\infty}_0(\bR_+)$ satisfies (\ref{11.4.1}) if   $\zeta>0$ on $[e^{-1},e]$.

Now, for $k\in\bZ$ and $x\in \cD$  we define
$$
\zeta_{k}(x)=\zeta(e^{k}\psi(x)).
$$
Then    $supp(\zeta_k) \subset G'_k:=\{x\in \cD:
e^{-k-k_0}<\rho(x)<e^{-k+k_0}\}$ with some integer $k_0>0$, 
\begin{equation*}
                                   \label{eqn 07.29.2}
\sum_{k=-\infty}^{\infty}\zeta_{k}(x)\geq \delta >0,
\end{equation*}
and
\begin{equation*}
                                                         \label{10.10.06}
\zeta_{k} \in C^{\infty}_0(G'_k), \quad |D^m \zeta_k(x)|\leq
N(\zeta,m) e^{mk}.
\end{equation*}


The following Lemma suggests us alternative definitions of  $H^n_{p,\Theta}(\cD)$ and  $K^n_{p,\theta,\Theta}(\cD)$. From now on, if a function  defined on $\cD$ vanishes near the boundary of $\cD$, then by a trivial extension we consider it as a function defined on $\bR^d$. 
Let $H^n_p:=W^n_p(\bR^d)$, the usual Sobolev space on $\bR^d$ (see Introduction for notation).

\begin{lemma}  \label{lemma 4.23.1}

Let $p\in(1,\infty)$, $\theta\in\bR$, $\Theta\in \bR$, and $n\in\{0,1,2,\cdots\}$.

\begin{enumerate}[align=right,label=\textup{(\roman*)}]
\item For any $\eta\in C^{\infty}_c(\bR_+)$ and $f\in H^{n}_{p,\Theta}(\cD)$,
\begin{equation*}
 \label{eqn 4.24.5}
  \sum_{k \in\bZ}
e^{k\Theta} \|\eta(e^{-k}\psi(e^k\cdot))f(e^{k}\cdot)\|^p_{H^{n}_p} \leq N(p,\Theta,d,n, \eta) \|f\|_{H^{n}_{p,\Theta}(\cD)}^{p}.
\end{equation*}

\item The reverse inequality of (i) also holds if $\eta$ satisfies (\textsl{\ref{11.4.1}}).

\item $f\in K^{n}_{p,\theta,\Theta}(\cD)$ if any only if $|x|^{(\theta-\Theta)/p}f\in H^n_{p,\Theta}(\cD)$ and
$$
\|f\|_{ K^{n}_{p,\theta,\Theta}(\cD)}  \quad \sim  \quad \||\cdot|^{(\theta-\Theta)/p}f(\cdot)\|_{H^n_{p,\Theta}(\cD)},
$$
where the equivalence relation depends only on $\theta,\Theta,n,\cM$.

\end{enumerate}

\end{lemma}

\begin{proof}
(i)-(ii).  See \cite[Propositio 2.2]{Lo1} (or \cite[Lemma 1.4]{Krylov 1999-1}). Below we give a short proof for reader's convenience. If $n=0$, then by  the change of variables $e^kx \to x$,
$$
 \sum_{k\in\bZ}
e^{k\Theta} \|\eta(e^{-k}\psi(e^k\cdot))f(e^{k}\cdot)\|^p_{L_p} =\int_{\cD}   \left[\sum_{k=-\infty}^{\infty} e^{k(\Theta-d)}|\eta(e^{-k}\psi(x))|^p\right] |f(x)|^p dx. 
$$
Thus   to prove (i),  we only use the fact (see e.g. \cite[Remark 1.3]{Krylov 1999-1}) that for any 
$\eta\in C^{\infty}_c (\bR_+)$, 
$$
\sum_{k\in \bZ} e^{k(\Theta-d)}|\eta(e^{-k}\psi(x))|^p\leq N  \psi^{\Theta-d}(x),
$$
and the reverse inequality also holds if $\eta$ satisfies \eqref{11.4.1}.  The proof for  $n=1,2,\ldots$ is  almost the same  and  mainly based on the change of variables $e^kx\to x$. We leave the detail to  the reader.

(iii)  This follows from the fact $H^{n}_{p,\Theta}(\cD)=K^n_{p,\Theta,\Theta}(\cD)$ and Lemma \ref{property1} (ii).
\end{proof}

Those alternative definitions will help us prove the main result of this section, Theorem \ref{regularity.induction}.  One issue is that in the proof we need the negative space $K^{-1}_{p,\theta+p,\Theta+p}(\cD)$ to be defined ahead. 
So, we extend the definition of the spaces for all $n\in\bZ$.
For $n\in \{-1,-2,\cdots\}$, let us define $H^n_p$ as the dual space of $H^{-n}_q$, where $1/p+1/q=1$.

\begin{defn}
\label{defn 4.24}
Let $p\in (1,\infty)$, $\theta\in\bR$, $\Theta\in \bR$, and $n\in \bZ$.

(i) We let $H^n_{p,\Theta}(\cD)$ denote the class of all  distributions $f$ on $\cD$ such that 
$$
\|f\|^p_{H^n_{p,\Theta}(\cD)}:=\sum_{k\in \bZ} e^{k\Theta} \|\zeta(e^{-k}\psi(e^k\cdot))f(e^{k}\cdot)\|^p_{H^n_p}<\infty.
$$

(ii) We write $f\in K^n_{p,\theta,\Theta}(\cD)$ if and only if $|x|^{(\theta-\Theta)/p}f\in H^n_{p,\Theta}(\cD)$, and  define
$$
\|f\|_{ K^{n}_{p,\theta,\Theta}(\cD)}:= \||\cdot|^{(\theta-\Theta)/p}f(\cdot)\|_{H^n_{p,\Theta}(\cD)}
$$
with the newly defined $H^n_{p,\Theta}(\cD)$ in (i).
\end{defn}


Then we have the following properties available.

\begin{lemma}
 \label{lemma 4.24.1}
  Let $p\in(1,\infty)$,  $\theta\in\bR$, and $\Theta\in \bR$.
  
  \begin{enumerate}[align=right,label=\textup{(\roman*)}]
  \item The claims of Lemma  \ref{lemma 4.23.1}(i)-(ii) hold for any $n\in \bZ$.

  \item For any $\varepsilon>0$, and $n_1,n_2,n_3\in \bZ$ with $n_1< n_2< n_3$,
$$
\|f\|_{H^{n_2}_{p,\Theta}(\cD)}\leq \varepsilon \|f\|_{H^{n_3}_{p,\Theta}(\cD)}+N(\varepsilon)
\|f\|_{H^{n_1}_{p,\Theta}(\cD)},
$$
$$
\|g\|_{K^{n_2}_{p,\theta,\Theta}(\cD)}\leq \varepsilon \|g\|_{K^{n_3}_{p,\theta,\Theta}(\cD)}+N(\varepsilon)
\|g\|_{K^{n_1}_{p,\theta,\Theta}(\cD)},
$$
where $N(\varepsilon)=N(\varepsilon, n_i, p,d,\theta,\Theta,\cM)$.

\item For any $\mu\in \bR$ and $n\in \bZ$,
$$
\|\psi^{\mu}f\|_{H^n_{p,\Theta}(\cD)} \sim \|f\|_{H^n_{p,\Theta+\mu p}(\cD)}, \quad\|\psi^{\mu}f\|_{K^n_{p,\theta,\Theta}(\cD)} \sim \|f\|_{K^n_{p,\theta+\mu p,\Theta+\mu p}(\cD)},
$$
where $\psi$ is from (\textsl{\ref{eqn 2.26.1}}).

\item Let $n\in \bZ$ and  $|a|^{(0)}_n:=\sup_{\cD} \sum_{|\alpha|\leq |n|} \rho^{|\alpha|}|D^{\alpha}a|<\infty$, then
$$
\|af\|_{K^n_{p,\theta,\Theta}(\cD)}\leq N(n,p,d)|a|^{(0)}_n \|f\|_{K^n_{p,\theta,\Theta}(\cD)}.
$$

 \item  For any $n\in \bZ$,
  $$
  \|Df\|_{H^{n}_{p,\Theta+p}(\cD)}\leq N \|f\|_{H^{n+1}_{p,\Theta}(\cD)}, \quad     \|Dg\|_{K^{n}_{p,\theta+p,\Theta+p}(\cD)} \leq N
   \|g\|_{K^{n+1}_{p,\theta,\Theta}(\cD)},
  $$
  where $N=N(d,p,n,\theta,\Theta,\cM)$.
\end{enumerate}
\end{lemma}

\begin{proof}
(i)   See \cite[Proposition 2.2]{Lo1}.  We also remark that (i) and (ii) are proved in \cite{Krylov 1999-1} on $\bR^d_+=\{x^1>0\}$.  On $\bR^d_+$ we can take $\psi(x)=\rho(x)=x^1$, and therefore
$$
\zeta(e^{-k}\psi(e^kx))=\zeta(x), \quad \eta(e^{-k}\psi(e^kx))=\eta(x).
$$
For an alternative proof of  (i)-(ii) on the conic domain $\cD$, it is enough to  replace $x^1$ by our $\psi$ and repeat the proof of \cite[Lemma 1.4]{Krylov 1999-1} word for word.  All the arguments there go through due to \eqref{eqn 2.26.1}.

\vspace{1mm}

(ii) Obviously we only need to prove the  first assertion, and this assertion is an easy consequence of  Definition \ref{defn 4.24} and the embedding inequality 
\begin{equation}
  \label{eqn 4.26.1}
\|h\|_{H^{n_2}_p}\leq \varepsilon \|h\|_{H^{n_3}_p}+N(\varepsilon,d,p)\|h\|_{H^{n_1}_p}.
\end{equation}

(iii) Again we only  prove the first relation. Also, since $\mu,\Theta\in \bR$ are arbitrary, it suffices to prove
$$
\|\psi^{\mu}f\|_{H^n_{p,\Theta}(\cD)}\leq N \|f\|_{H^n_{p,\Theta+\mu p}(\cD)}.
$$
By definition,
\begin{eqnarray*}
\|\psi^{\mu}f\|^p_{H^n_{p,\Theta}(\cD)}&=&\sum_{k\in \bZ} e^{k\Theta} \|\psi^{\mu}(e^kx)\zeta(e^{-k}\psi(e^kx))f(e^{k}x)\|^p_{H^n_p}\\
&=&\sum_{k\in \bZ} e^{k(\Theta+\mu p)} \|e^{-k \mu}\psi^{\mu}(e^kx)\zeta(e^{-k}\psi(e^kx))f(e^{k}x)\|^p_{H^n_p}\\
&=&\sum_{k\in \bZ} e^{k(\Theta+\mu p)} \|e^{-k \mu}\psi^{\mu}(e^kx) \eta(e^{-k}\psi(e^kx))\zeta(e^{-k}\psi(e^kx))f(e^{k}x)\|^p_{H^n_p},
\end{eqnarray*}
where  $\eta \in C^{\infty}_c(\bR_+)$ such that $\eta=1$ on the support of $\zeta$. Note that  $\psi^{\mu}(e^kx)\sim e^{k\mu}$ on the support of $\eta(e^{-k}\psi(e^kx))$, and moreover using this and \eqref{eqn 2.26.1}  one can check
\begin{eqnarray*}
&&|e^{-k \mu}\psi^{\mu}(e^kx) \eta(e^{-k}\psi(e^kx))|_n\\
&:=&\sum_{|\alpha|\leq |n|} \sup_x |D^{\alpha} \left(e^{-k \mu}\psi^{\mu}(e^kx) \eta(e^{-k}\psi(e^kx))\right)|\leq N(n)<\infty.
\end{eqnarray*}
To prove (iii),  we only need to use the classical result
\begin{equation}
  \label{eqn 4.24.6}
\|af\|_{H^n_p}\leq N(d,p,n)|a|_n \|f\|_{H^n_p}, \quad n\in \bZ.
\end{equation}

(iv) This can be  proved as in the proof of (iii)  using \eqref{eqn 4.24.6}.

(v)   By definition
$$
\|f_x\|^p_{H^{n}_{p,\Theta+p}(\cD)}=\sum_{k\in \bZ} e^{k(\Theta+p)} \|\zeta(e^{-k}\psi(e^kx))f_x(e^{k}x)\|^p_{H^{n}_p}.
$$
Note 
$$
\zeta(e^{-k}\psi(e^kx)) f_x(e^kx)= e^{-k}  (\zeta(e^{-k}\psi(e^kx))f(e^kx))_x-e^{-k} \zeta_x (e^{-k}\psi(e^kx))f(e^kx).
$$
Since $D:H^{n+1}_p\to H^{n}_p$ is bounded and $\|\cdot\|_{H^{n}_p} \leq \|\cdot\|_{H^{n+1}_p}$,
$$
\|f_x\|^p_{H^{n}_{p,\Theta+p}(\cD)} \leq N \sum_{k\in \bZ} e^{k\Theta} \left( \|\zeta(e^{-k}\psi(e^kx))f(e^kx)\|^p_{H^{n+1}_p}+\|\zeta_x (e^{-k}\psi(e^kx))f(e^kx)\|^p_{H^{n+1}_p}\right).
$$
This and the result of (i) prove  the first assertion of $(v)$.   For the second assertion, we denote $\xi:=|x|^{(\theta-\Theta)/p}$ and  observe 
\begin{eqnarray*}
  \|g_x\|_{K^{n}_{p,\theta+p,\Theta+p}(\cD)}&:=&
  \| \xi g_x\|_{H^{n}_{p,\Theta+p}(\cD)} =\|(\xi g)_x-g \xi_x\|_{H^{n}_{p,\Theta+p}(\cD)}  \\
  &\leq& N \|\xi g\|_{H^{n+1}_{p,\Theta}(\cD)}+\|g \xi_x \|_{H^n_{p,\Theta+p}(\cD)}\\
  &\leq&N \| g\|_{K^{n+1}_{p,\Theta}(\cD)}+ N \|g \xi^{-1}\xi_x\|_{K^n_{p,\Theta+p}(\cD)} \\
  &\leq&N \| g\|_{K^{n+1}_{p,\Theta}(\cD)}+ N \|g \left(\psi \xi^{-1}\xi_x\right)\|_{K^n_{p,\Theta}(\cD)}.
   \end{eqnarray*}
   The last inequality above is due to (iii). 
Also, by \eqref{estimate.rho}
  \begin{equation}
    \label{eqn 4.25.1}
 \big| \psi \xi^{-1}\xi_x  \big|^{(0)}_n <\infty,
  \end{equation}
  and therefore it is enough to apply the result of (iv). The lemma is proved.
\end{proof}

\begin{corollary}  \label{main corollary}

Let $p\in(1,\infty)$, $\theta\in\bR$, $\Theta\in \bR$, and $n\in \bZ$.  Put $\xi(x)=|x|^{(\theta-\Theta)/p}$.  Then
$$
\|\xi^{-1}g \xi_{x}\|_{K^n_{p,\theta+p,\Theta+p}(\cD)}\leq N \|g\|_{K^n_{p,\theta,\Theta}(\cD)},
$$
and
$$
\|\xi^{-1}f \xi_{xx}\|_{K^n_{p,\theta+p,\Theta+p}(\cD)}\leq N \|f\|_{K^n_{p,\theta-p,\Theta-p}(\cD)},
$$
where $N=N(d,p,\theta,\Theta,n,\cM)$.  

\end{corollary}

\begin{proof}
Due to the similarity, we only prove the first assertion. By Lemma \ref{lemma 4.24.1} (iii),
$$
\|\xi^{-1}g \xi_{x}\|_{K^n_{p,\theta+p,\Theta+p}(\cD)}\leq N \|\psi \xi^{-1} \xi_x g\|_{K^n_{p,\theta,\Theta}(\cD)}.
$$
The assertion follows from \eqref{eqn 4.25.1} and  Lemma \ref{lemma 4.24.1}
 (iv).
\end{proof}

The main result of this section,  Theorem \ref{regularity.induction} below, is based on Definition \ref{defn 4.24}  and the following result on $\bR^d$.
For $n\in \bZ$ we denote
$$
\bH^n_p(T):=L_p((0,T); H^n_p), \quad   \bL_p(T):=\bH^0_p(T).
$$
\begin{lemma} \label{lemma entire}
 Let $p\in(1,\infty)$ and $n\in \{0,1,2,\cdots\}$. Also, let $f\in \bH^{n-1}_p(T)$ and $u\in \bL_p(T)$ satisfy
\begin{align*}
u_t=\cL u+f, \quad t\in(0,T] \quad ; \quad u(0,\cdot)=0
\end{align*}
in the sense of distributions on $\bR^d$. Then $u\in \bH^{n+1}_p(T)$ and 
 $$
 \|u\|_{\bH^{n+1}_p(T)}\leq N \|u\|_{\bL_p(T)}+N\|f\|_{\bH^{n-1}_p(T)},
 $$
 where $N=N(d,p,\nu_1,\nu_2)$ is independent of $T$.
 \end{lemma} 
 
 \begin{proof}
By e.g. \cite[Theorem 1.1]{Krylov 2002},
 $$
 \|u_{xx}\|_{\bH^{n-1}_p(T)}\leq N(d,p,\nu_1,\nu_2) \|f\|_{\bH^{n-1}_p(T)}.
 $$
 This,  \eqref{eqn 4.26.1}, and the inequality
 $$
 \|u\|_{\bH^{n+1}_p(T)}\leq N(\|u_{xx}\|_{\bH^{n-1}_p(T)}+\|u\|_{\bH^{n-1}_p(T)} )
 $$
 yield the claim of the lemma. 
  \end{proof}

 The following is the main result of this section. One can refer to Section  \ref{sec:Cone} for the definitions of the function spaces appearing in the statement. We remark that the theorem holds for any $\theta,\Theta\in \bR$.

\begin{thm}\label{regularity.induction}
Let $p\in(1,\infty)$, $\theta\in\bR$, $\Theta \in\bR$,  and $n\in\{0,1,2,\cdots\}$.  Assume that $f\in\bK^{n}_{p,\theta+p,\Theta+p}(\cD,T)$ and  $u\in\bL_{p,\theta-p,\Theta-p}(\cD,T)$ satisfies
\begin{align*}
u_t=\cL u+f, \quad t\in(0,T] \quad ; \quad u(0,\cdot)=0
\end{align*}
in the sense of distributions on $\cD$. Then $u\in\bK^{n+2}_{p,\theta-p,\Theta-p}(\cD,T)$, hence $u\in \cK^{n+2}_{p,\theta,\Theta}(\cD,T)$, and  the estimate
\begin{align}\label{regularity.est.}
\|u\|_{\bK^{n+2}_{p,\theta-p,\Theta-p}(\cD,T)}\leq N\left(\|u\|_{\bL_{p,\theta-p,\Theta-p}(\cD,T)}+\|f\|_{\bK^{n}_{p,\theta+p,\Theta+p}(\cD,T)}\right)
\end{align}
holds, where $N=N(\cM,p,n,\theta,\Theta,\nu_1,\nu_2)$ which is  in particular independent of $f$, $u$, and $T$.
\end{thm}

\begin{proof}

\textbf{1}.  Assume  $u\in \bK^{m}_{p,\theta-p,\Theta-p}(\cD,T)$,  $m \in \{0,1,2,\cdots,n+1\}$. We prove that
$u\in \bK^{m+1}_{p,\theta-p,\Theta-p}(\cD,T)$
and the estimate 
$$
\|u\|_{\bK^{m+1}_{p,\theta-p,\Theta-p}(\cD,T)}\leq N\left(\|u\|_{\bK^m_{p,\theta-p,\Theta-p}(\cD,T)}+\|f\|_{\bK^{m-1}_{p,\theta+p,\Theta+p}(\cD,T)}\right)
$$
holds with $N=N(\cM,p,n,\theta,\Theta,\nu_1,\nu_2)$. 

Put $\xi(x)=|x|^{(\theta-\Theta)/p}$. Using Definition  \ref{defn 4.24} and the change of variables $t\to e^{2k}t$, we have 
\begin{eqnarray}
&&\quad \|u\|^p_{\bK^{m+1}_{p,\theta-p,\Theta-p}(\cD,T)} \nonumber\\
&&= \sum_{k\in \bZ} e^{k(\Theta-p)}\|\zeta(e^{-k}\psi(e^k\cdot))\xi(e^{k}\cdot)u(\cdot,e^k\cdot)\|^p_{\bH^{m+1}_p(T)} \nonumber\\
  &&= \sum_{k\in \bZ} e^{k(\Theta-p+2)}\|\zeta(e^{-k}\psi(e^k\cdot))\xi(e^{k}\cdot)u(e^{2k}\cdot,e^k\cdot)\|^p_{\bH^{m+1}_p(e^{-2k}T)}.   \label{eqn 4.23.1}  
   \end{eqnarray}

Denote $v_k(t,x):=\zeta(e^{-k}\psi(e^kx))\xi(e^{k}x)u(e^{2k}t,e^kx)$, then $v_k$ satisfies 
 $$
 (v_k)_t=\cL_k v_k +f_k, \quad  t\leq e^{-2k}T\quad ; \quad v_k(0,\cdot)=0
 $$
 in the sense of distributions on $\bR^d$,
 where
 $$
 \cL_k:=\sum_{i,j} a^{ij}_k(t)D_{ij},\quad a^{ij}_k(t):=a^{ij}(e^{2k}t)
 $$
 and, with Einstein's summation convention with respect to $i,j$,
 \begin{eqnarray*}
 f_k(t,x)&:=&\quad e^{2k}\zeta(e^{-k}\psi(e^kx))\xi(e^{k}x)f(e^{2k}t,e^kx) \\
 &&+e^{k} a^{ij}_k(t)D_iu(e^{2k}t,e^kx) \zeta'(e^{-k}\psi(e^kx)) D_j\psi(e^kx)  \xi(e^kx)\\
 &&+e^{2k}a^{ij}_k(t)D_iu(e^{2k}t,e^kx) \zeta(e^{-k}\psi(e^kx)) D_j \xi(e^kx) \\
 &&+e^ka^{ij}_k(t)u(e^{2k}t,e^kx)\zeta'(e^{-k}\psi(e^kx))  D_i\psi(e^kx)D_j\xi(e^kx)\\
 &&+e^{2k}a^{ij}_k(t)u(e^{2k}t,e^kx)\zeta(e^{-k}\psi(e^kx))D_{ij}\xi(e^kx) \\
 &&+a^{ij}_k(t)u(e^{2k}t,e^kx)\zeta''(e^{-k}\psi(e^kx))D_i\psi(e^kx) D_j\psi(e^kx) \xi(e^kx)\\
 &&+ e^ka^{ij}_k(t)u(e^{2k}t,e^kx)\zeta'(e^{-k}\psi(e^kx))  D_{ij}\psi(e^kx)\xi(e^kx)\\
 &=:& \sum_{l=0}^6 f^{l}_k(t,x)
\end{eqnarray*}
;  $\zeta'$ and $\zeta''$ denote the first and second derivative of $\zeta$, respectively.  We note that the operator $\cL_k$ for any $k\in\bZ$ satisfies  the uniform parabolicity condition \eqref{uniform parabolicity}.
 
 We can apply Lemma \ref{lemma entire} and from  \eqref{eqn 4.23.1} we get
 \begin{eqnarray*}
  \|u\|^p_{\bK^{m+1}_{p,\theta-p,\Theta-p}(\cD,T)} &\leq& N   \sum_{k\in \bZ} e^{k(\Theta-p+2)}\|v_k\|^p_{\bL_p(e^{-2k}T)}\\
  &&+N \sum_{l=0}^6  \sum_{k\in \bZ} e^{k(\Theta-p+2)}\|f^{l}_k\|^p_{\bH^{m-1}_p(e^{-2k}T)}
  \end{eqnarray*}
if
 $$
 v_k \in \bL_p(e^{-2k}T), \quad f^l_k \in \bH^{m-1}_p(e^{-2k}T),\quad k\in\bZ, \; l=0,1,\ldots,6,
 $$
 are provided ahead.
 
Indeed, the change of variable $e^{2k}t \to t$ and Definition \ref{defn 4.24} yield
  \begin{eqnarray*}
&&   \sum_{k\in \bZ} e^{k(\Theta-p+2)}\|v_k\|^p_{\bL_p(e^{-2k}T)}\\
&=& \sum_{k\in \bZ} e^{k(\Theta-p)}\| \zeta(e^k\psi(e^k\cdot))\xi(e^{k}\cdot)u(\cdot,e^k\cdot) \|^p_{\bL_p(T)}=  \|u\|^p_{\bL_{p,\theta-p,\Theta-p}(\cD,T)},
   \end{eqnarray*}
meaning especially $v_k \in \bL_p(e^{-2k}T)$ for all $k$. 

Next we show $f^l_k \in \bH^{m-1}_p(e^{-2k}T)$ as follows. For $l=0$,  by Definition \ref{defn 4.24} and the change of variable $e^{2k}t \to t$, we have
   \begin{eqnarray*}
 &&   \sum_{k\in \bZ} e^{k(\Theta-p+2)}\|f^0_k\|^p_{\bH^{m-1}_p(e^{-2k}T)}\\
 && =  
    \sum_{k\in \bZ} e^{k(\Theta+p)}\|\zeta(e^{-k}\psi(e^k\cdot))\xi(e^{k}\cdot)f(e^{2k}\cdot,e^k\cdot) \|^p_{\bH^{m-1}_p(T)} = \|f\|^p_{\bK^{m-1}_{p,\theta+p,\Theta+p}(\cD,T)}.
    \end{eqnarray*}
For $l=1$,  by Definition \ref{defn 4.24} and  Lemma \ref{lemma 4.24.1} (i), we get 
     \begin{eqnarray*}
 && \quad \sum_{k\in \bZ} e^{k(\Theta-p+2)}\|f^1_k\|^p_{\bH^{m-1}_p(e^{-2k}T)}\\
 &&\leq N   
    \sum_{k\in \bZ} \sum_{i,j}e^{k\Theta}\|D_iu(\cdot,e^k\cdot)\zeta'(e^{-k}\psi(e^k\cdot))D_j\psi(e^k\cdot)\xi(e^k\cdot)\|^p_{\bH^{m-1}_p(T)} \\
    &&\leq N \|\psi_x u_x  \xi\|^p_{\bH^{m-1}_{p,\Theta}(\cD,T)}\\
    &&=N \|\psi_x u_x\|^p_{\bK^{m-1}_{p,\theta,\Theta}(\cD,T)}
    \leq 
    N \|u_x\|^p_{\bK^{m-1}_{p,\theta,\Theta}(\cD,T)}\leq 
    N \|u\|^p_{\bK^{m}_{p,\theta-p,\Theta-p}(\cD,T)},
   \end{eqnarray*}
 where  the   last two inequalities are  due to    \eqref{eqn 2.26.1}, Lemma \ref {lemma 4.24.1} (iv), and Lemma  \ref {lemma 4.24.1} (v).
For $l=2$, by definitions of norms, we have
      \begin{eqnarray*}
  && \sum_{k\in \bZ} e^{k(\Theta-p+2)}\|f^2_k\|^p_{\bH^{m-1}_p(e^{-2k}T)}  \\
  & \leq& N  \sum_{k\in \bZ} \sum_{i,j}e^{k(\Theta+p)}\|D_iu(\cdot,e^k\cdot) \zeta(e^{-k}\psi(e^k\cdot)) D_j\xi(e^k\cdot)\|^p_{\bH^{m-1}_p(T)}   \\
  &\le&N
   \|u_x \xi_x\|^p_{\bH^{m-1}_{p,\theta+p,\Theta+p}(\cD,T)}
   = N 
    \|u_x \xi^{-1}\xi_x\|^p_{\bK^{m-1}_{p,\theta+p,\Theta+p}(\cD,T)}   \leq N\|u_x\|^p_{\bK^{m-1}_{p,\theta,\Theta}(\cD,T)},
   \end{eqnarray*}
   where the  last  inequality is due to Corollary \ref{main corollary}. For other $l$ one can argue similarly. We proved $f^l_k \in \bH^{m-1}_p(e^{-2k}T)$ for all $k$ and $l$. Moreover, we additionally proved
  \begin{eqnarray*}
&&  \sum_{l=0}^6  \sum_{k\in \bZ} e^{k(\Theta-p+2)}\|f^{l}_k\|^p_{\bH^{m-1}_p(e^{-2k}T)}\\
&\leq& N\ \|u\|^p_{\bK^{m}_{p,\theta-p,\Theta-p}(\cD,T)}+ N\|f\|^p_{\bK^{m-1}_{p,\theta+p,\Theta+p}(\cD,T)}.  
  \end{eqnarray*}

   Consequently, we have
   \begin{eqnarray*}
    \|u\|^p_{\bK^{m+1}_{p,\theta-p,\Theta-p}(\cD,T)} \leq N\|u\|^p_{\bK^m_{p,\theta-p,\Theta-p}(\cD,T)}+N\|f\|^p_{\bK^{m-1}_{p,\theta+p,\Theta+p}(\cD,T)} 
    \end{eqnarray*} 
and this is exactly we wanted to prove. 

\textbf{2}. By  applying the result of Step 1 for $m=0,1,\cdots, n+1$  inductively and by the fact that $L_p(\bR^d)$ is continuously embedded in $H^{-1}_p(\bR^d)$, we conclude that $u$ belongs to $\bK^{n+2}_{p,\theta-p,\Theta-p}(\cD,T)$ and  \eqref{regularity.est.} holds. This  ends the proof.
 \end{proof}

 \mysection{The proof of Theorem~\ref{main result}}\label{sec:main result}
 
We start with a representation formula of the solution.
\begin{lemma}\label{sol.rep.}
Assume that $p\in(1,\infty)$, $p(1-\lambda^+_{c,\cL})<\theta<p(d-1+\lambda^-_{c,\cL})$, and $\Theta\in (d-1,d-1+p)$.  Let $f\in\bL_{p,\theta+p,\Theta+p}(\cD,T)$ and let $u\in\cK^2_{p,\theta,\Theta}(\cD,T)$ be a solution to equation (\textsl{\ref{heat eqn}}) with the source term $f$.  Then   $u=v$ in $\bL_{p,\theta-p,\Theta-p}(\cD,T)$, where $v$ is the function defined by
\begin{align}
v(t,x)=\int^t_0\int_{\cD}G(t,s,x,y)f(s,y)dyds.\label{solution representation}
\end{align}
\end{lemma}

\begin{proof}

Take  a sequence of functions $\xi_m$ from  the proof of Lemma \ref{property1},  and choose $\eta\in C^{\infty}_0(\bR^d)$ such that $0\leq \eta\leq 1$ and $\eta(x)=1$ for $|x|\leq 1$.  Denote
$$
\eta_m(x):=\xi_m(x)\eta(x/m).
$$
Then, by the choice of $\xi_m$, one can easily check that $\eta_m\in C^{\infty}_c(\cD)$,  $0\leq \eta_m\leq 1$, $\eta_m(x)\to 1$ as $m\to \infty$ for $x\in \cD$, and  $\rho^{|\beta|}D^{\beta}\eta_m$ is uniformly bounded  and goes to zero as $m\to \infty$ for any mullti-index $\beta$ with $|\beta|\geq 1$.

Denote $u_m=u\eta_m$, then it satisfies
$$
(u_m)_t=\cL u_m +f_m+f\eta_m, \quad t\in(0,T]\quad ; \quad  u_m(0,\cdot)=0
$$
in the sense of distributions on $\cD$,
where
$$
f_m:=f\eta_m+a^{ij}u_{x^i}(\eta_{m})_{x^j}+a^{ij}u (\eta_{m})_{x^ix^j}
$$
with Einstein's summation notation used on $i,j$.

Note that, since $\eta_m$ has compact support, we have
\begin{equation*}
 \label{eqn 4.25.7}
\int^T_0\int_{\cD} \left(|\rho_0 (u_{m})_{xx}|^p+|(u_{m})_{x}|^p+|\rho^{-1}_0u_m|^p+|\rho_0 f_m|^p \right) \rho_0^{\theta-d} dxdt<\infty.
\end{equation*}
Thus, by Theorem 1.1 and Theorem 4.1 in  \cite{Kozlov Nazarov 2014}, we have representation of $u_m$, that is,
$$
u_m(t,x)=\int^t_0\int_{\cD}G(t,s,x,y) f_m(s,y)dyds,
$$
which gives
\begin{eqnarray*}
&&u_m(t,x)-v(t,x)\\
&=& \int^t_0\int_{\cD}G(t,s,x,y) \left((1-\eta_m)f+a^{ij}u_{x^i}(\eta_{m})_{x^j}+a^{ij}u(\eta_{m})_{x^ix^j}\right) dyds.
\end{eqnarray*}
Consequently, by Lemma \ref{main est} and Lemma \ref{lemma 4.24.1} (iii), we get
\begin{eqnarray*}
&&\|u_m-v\|_{\bL_{p,\theta-p,\Theta-p}(\cD,T)}\\
&\leq& N \|(1-\eta_m)f+a^{ij}u_{x^i}(\eta_{m})_{x^j}+a^{ij}u(\eta_{m})_{x^ix^j}\|_{\bL_{p,\theta+p,\Theta+p}(\cD,T)}\\
&\leq& N\|(1-\eta_m)f\|_{\bL_{p,\theta+p,\Theta+p}(\cD,T)}+ N\|u_x \psi (\eta_{m})_{x}\|_{\bL_{p,\theta,\Theta}(\cD,T)}\\
&&  +N\|u \psi^2 (\eta_{m})_{xx}\|_{\bL_{p,\theta-p,\Theta-p}(\cD,T)}.
\end{eqnarray*}
Since $\psi D_i\eta_{m}$ and $\psi^2 D_{ij}\eta_{m}$ are uniformly bounded and go to zero as $n\to \infty$, we conclude that the last three terms above go to zero, and therefore $u_m \to v$ in $\bL_{p,\theta-d,\Theta-d}(\cD,T)$ as $n\to \infty$. This and the fact $u_m \to u$ in $\bL_{p,\theta-d,\Theta-d}(\cD,T)$ finish the proof of the lemma.
\end{proof}

We are ready to prove our main result. 
\vspace{0.3cm}
\begin{proof}[Proof of Theorem~\ref{main result}] \quad
\vspace{0.2cm} 

  \textbf{Existence and the estimate}.
  For the given $f\in\bK^n_{p,\theta+p,\Theta+p}(\cD,T)$, choose functions $f_m\in \cC^{\infty}_c((0,T)\times\cD)$ such that  $f_m\rightarrow f$ as $m\to\infty$ in $\bK^n_{p,\theta+p,\Theta+p}(\cD,T)$.
Then, by  \cite[Theorem 4.1]{Kozlov Nazarov 2014}, the function
$$
u_m(t,x):=\int^t_0\int_{\cD}G(t,x,s,y)f_m(s,y)dy ds
$$
is a solution to \eqref{heat eqn} with the source term $f_m$. 

Now,  Lemma~\ref{main est} comes in and, for any $m,m'\in\bN$, it says
$$
\|u_m\|_{\bL_{p,\theta-p,\Theta-p}(\cD,T)}\leq N \|f_m\|_{\bL_{p,\theta+p,\Theta+p}(\cD,T)}
$$
and
$$
\|u_m-u_{m'}\|_{\bL_{p,\theta-p,\Theta-p}(\cD,T)}\leq N \|f_m-f_{m'}\|_{\bL_{p,\theta+p,\Theta+p}(\cD,T)}.
$$
In particular, $u_m\in\bL_{p,\theta-p,\Theta-p}(\cD,T)$. We emphasize that the dependency of $N$ changes  by the situations mentioned in Lemma ~\ref{main est}.  Then,  thanks to Theorem \ref{regularity.induction}, we have
\begin{align}
&\|u_m\|_{\cK^{n+2}_{p,\theta-p,\Theta-p}(\cD,T)}\nonumber \\
&=\|u_m-u\|_{\bK^{n+2}_{p,\theta-p,\Theta-p}(\cD,T)}+\|(u_m)_t\|_{\bK^{n}_{p,\theta+p,\Theta+p}(\cD,T)}\nonumber\\
&\leq N\left(\|u_m\|_{\bL_{p,\theta-p,\Theta-p}(\cD,T)}+\|f_m\|_{\bK^n_{p,\theta+p,\Theta+p}(\cD,T)}\right)\nonumber\\
&\leq  N\|f_m\|_{\bK^n_{p,\theta+p,\Theta+p}(\cD,T)}.\label{for main estimate}
\end{align}
Similarly, we have
\begin{align*}
\|u_m-u_{m'}\|_{\cK^{n+2}_{p,\theta-p,\Theta-p}(\cD,T)}
\leq  N\|f_m-f_{m'}\|_{\bK^n_{p,\theta+p,\Theta+p}(\cD,T)}.
\end{align*}
 It follows that $u_m$ is a Cauchy sequence in $\cK^{n+2}_{p,\theta,\Theta}(\cD,T)$ and there exists $u\in \cK^{n+2}_{p,\theta,\Theta}(\cD,T)$ such that $u_m\rightarrow u$ in $\cK^{n+2}_{p,\theta,\Theta}(\cD,T)$.
Moreover, since
\begin{align*}
\cL u_m\rightarrow \cL u, \quad (u_m)_t \rightarrow u_t \quad \text{as}\quad m \rightarrow \infty \quad \text{in}\quad\bK^n_{p,\theta+p,\Theta+p}(\cD,T),
\end{align*}
we have $u_t=\cL u+f$ in $\bK^n_{p,\theta+p,\Theta+p}(\cD,T)$. This handles the existence of a solution to equation \eqref{heat eqn} in $\cK^{n+2}_{p,\theta,\Theta}(\cD,T)$. Moreover, estimate \eqref{main estimate}  follows from \eqref{for main estimate}.

\textbf{Uniqueness}.
 Let $w\in \cK^{n+2}_{p,\theta,\Theta}(\cD,T)$ be a solution  to equation \eqref{heat eqn} with $f\equiv 0$. Then by Lemma \ref {sol.rep.}, $w$ coincides with $v$ defined in \eqref{solution representation}, which is now identically zero. This handles the uniqueness. The theorem is proved. 
\end{proof}

\providecommand{\bysame}{\leavevmode\hbox to3em{\hrulefill}\thinspace}
\providecommand{\MR}{\relax\ifhmode\unskip\space\fi MR }
\providecommand{\MRhref}[2]{%
  \href{http://www.ams.org/mathscinet-getitem?mr=#1}{#2}
}
\providecommand{\href}[2]{#2}

\end{document}